\documentclass{amsart}
\usepackage[utf8]{inputenc}
\usepackage{todonotes} 
\usepackage{amsmath, amssymb, amsthm} 
\usepackage{latexsym} 
\usepackage{dsfont}
\usepackage[flushleft]{paralist}
\usepackage[all]{xy}
\usepackage{tikz-cd}
\usepackage{stmaryrd} 
\usepackage{hyperref} 
\usepackage[OT2,T1]{fontenc}
\tikzcdset{scale cd/.style={every label/.append style={scale=#1},
    cells={nodes={scale=#1}}}}
\DeclareSymbolFont{cyrletters}{OT2}{wncyr}{m}{n}
\DeclareMathSymbol{\Zhe}{\mathalpha}{cyrletters}{"11} 
    
\begin{document}
\title[Non-commutative main conjecture for graphs]{On the non-commutative Iwasawa main conjecture for voltage covers of graphs}
\author[S.~Kleine]{Sören Kleine} 
\address[Kleine]{Institut für Theoretische Informatik, Mathematik und Operations Research, Universität der Bundeswehr München, Werner-Heisenberg-Weg 39, 85577 Neubiberg, Germany} 
\email{soeren.kleine@unibw.de} 
\author[K.~Müller]{Katharina Müller}
\address[Müller]{D\'epartement de Math\'ematiques et de Statistique\\
Universit\'e Laval, Pavillion Alexandre-Vachon\\
1045 Avenue de la M\'edecine\\
Qu\'ebec, QC\\
Canada G1V 0A6}
\email{katharina.mueller.1@ulaval.ca}


\newcommand{\R}{\mathds{R}}
 	\newcommand{\Z}{\mathds{Z}}
 	\newcommand{\N}{\mathds{N}}
 	\newcommand{\Q}{\mathds{Q}}
 	\newcommand{\K}{\mathds{K}}
 	\newcommand{\M}{\mathds{M}} 
 	\newcommand{\C}{\mathds{C}}
 	\newcommand{\B}{\mathds{B}}
 	\newcommand{\LL}{\mathds{L}}
 	\newcommand{\F}{\mathds{F}}
 	\newcommand{\p}{\mathfrak{p}} 
 	\newcommand{\q}{\mathfrak{q}} 
 	\newcommand{\f}{\mathfrak{f}} 
 	\newcommand{\Pot}{\mathcal{P}}
 	\newcommand{\Gal}{\textup{Gal}}
 	\newcommand{\rg}{\textup{rank}}
 	\newcommand{\id}{\textup{id}}
 	\newcommand{\Ker}{\textup{Ker}}
 	\newcommand{\Image}{\textup{Im}} 
 	\newcommand{\la}{\langle}
 	\newcommand{\ra}{\rangle}
 	\newcommand{\gdw}{\Leftrightarrow}
 	\newcommand{\pfrac}[2]{\genfrac{(}{)}{}{}{#1}{#2}}
 	\newcommand{\Ok}{\mathcal{O}}
 	\newcommand{\Norm}{\mathrm{N}} 
 	\newcommand{\coker}{\mathrm{coker}}
 	\newcommand{\dotcup}{\stackrel{\textstyle .}{\bigcup}}
 	\newcommand{\Cl}{\mathrm{Cl}}
 	\newcommand{\Sel}{\textup{Sel}} 
 	\newcommand{\OkG}{\Ok\llbracket G\rrbracket }
 	\newcommand{\Ann}{\textup{Ann}} 
    \newcommand{\Div}{\textup{Div}}
    \newcommand{\on}{\textup{ on }}
    \newcommand{\Nrd}{\textup{Nrd}} 

 	\newtheorem{lemma}{Lemma}[section] 
 	\newtheorem{prop}[lemma]{Proposition} 
 	\newtheorem{defprop}[lemma]{Definition and Proposition} 
 	\newtheorem{conjecture}[lemma]{Conjecture} 
 	\newtheorem{thm}[lemma]{Theorem} 
 	\newtheorem*{thm*}{Theorem} 
 	\newtheorem{corollary}[lemma]{Corollary}
 	\newtheorem{claim}[lemma]{Claim}

 	\theoremstyle{definition}
 	\newtheorem{def1}[lemma]{Definition} 
 	\newtheorem{ass}[lemma]{Assumption}
 	\newtheorem{rem}[lemma]{Remark} 
 	\newtheorem{rems}[lemma]{Remarks} 
 	\newtheorem{example}[lemma]{Example} 
 	\newtheorem{fact}[lemma]{Fact} 

\begin{abstract} 
Let $p$ be a rational prime, and let $X$ be a connected finite graph. In this article we study voltage covers $X_\infty$ of $X$ attached to a voltage assignment ${\alpha}$ which takes values in some uniform $p$-adic Lie group $G$. We formulate and prove an Iwasawa main conjecture for the projective limit of the Picard groups $\textup{Pic}(X_n)$ of the intermediate voltage covers $X_n$, ${n \in \N}$, and we prove one inclusion of a main conjecture for the projective limit of the Jacobians $J(X_n)$.  

Moreover, we study the $\mathfrak{M}_H(G)$-property of $\Z_p\llbracket G\rrbracket $-modules and prove a necessary condition for this property which involves the $\mu$-invariants of $\Z_p$-subcovers ${Y \subseteq X_\infty}$ of $X$. If the dimension of $G$ is equal to 2, then this condition is also sufficient. 
\end{abstract} 

\maketitle 
\section{Introduction} 
In Iwasawa theory one studies the asymptotic growth of arithmetic objects (such as ideal class groups, Selmer groups, Tate-Shafarevic groups etc.) in suitable infinite pro-$p$ extensions of (number) fields. The Galois group of such an extension is usually a compact pro-$p$, $p$-adic Lie group $G$. For example, if $K_\infty$ denotes a so-called \emph{$\Z_p$-extension} of a number field $K$ (this means that $K_\infty$ is a Galois extension of $K$ and ${\Gal(K_\infty/K)}$ is isomorphic to $\Z_p$, the additive group of $p$-adic integers), then for each ${n \in \N}$ there exists a unique intermediate number field ${K_n \subseteq K_\infty}$ of degree $p^n$ over $K$, and Iwasawa has shown in \cite{iwasawa} that the exponent $e_n$ of the maximal $p$-power dividing the class number of $K_n$ grows via an asymptotic formula of the following form: 
\begin{align} \label{eq:iwasawa} e_n = \mu p^n + \lambda n + \nu\end{align}  
for each sufficiently large $n$ and fixed constants ${\mu, \lambda, \nu \in \Z}$. 
The major insight of Iwasawa, which enabled him to prove this famous result, was that the investigation of the ideal class groups of the $K_n$ becomes easier if one combines all of them to a so-called Iwasawa module over the completed group ring $\Z_p\llbracket G\rrbracket $, ${G = \Gal(K_\infty/K)}$ (i.e. by considering the projective limit of the ideal class groups over all intermediate number fields of the extension). As the structure of modules over such group rings is fairly well understood this enabled Iwasawa to derive equation~\eqref{eq:iwasawa}.

In this paper we will understand by an \emph{Iwasawa module} any finitely generated module $M$ over a completed group ring ${\Lambda := \Z_p\llbracket G\rrbracket }$, with $G$ as above (see Section~\ref{section:notation} for a more precise definition). One of the most prominent aims in Iwasawa theory is the formulation and proof of suitable \emph{Iwasawa main conjectures}. These conjectures describe a deep connection between the arithmetic-algebraic Iwasawa modules on the one hand and analytical objects (namely, the special values of suitable $p$-adic $L$-functions) on the other hand. The first such Iwasawa main conjecture was proved by Mazur and Wiles in \cite{mazur-wiles}. Over the last decades much research focussed on proving main conjectures in more general contexts. To name just a few results we refer to the groundbreaking papers \cite{skinner-urban}, \cite{ritter-weiss} and \cite{kakde}. 

A few years ago, Vallières initiated the Iwasawa theory \emph{of graphs} (see \cite{vallieres1}). Let $X_\infty$ be a $\Z_p$-cover of a finite connected graph $X$, i.e. $X_\infty$ is the "limit" of a sequence of Galois covers $X_n$ of $X$ such that ${\Gal(X_n/X) \cong \Z/p^n\Z}$ for each $n$. It turns out that the $p$-part of the number of \emph{spanning trees} of $X_n$ grows asymptotically in a way which is analogous to Iwasawa's growth formula~\eqref{eq:iwasawa}, i.e. it can again be described completely in terms of some invariants $\mu$, $\lambda$ and $\nu$ attached to the $\Z_p$-cover $X_\infty/X$. The proof of this result, which was given in increasing generality in a series of papers of Vallières and McGown (see \cite{vallieres1, vallieres2, vallieres3}), was completely analytic, in the sense that it was based on a direct relation of the number of spanning trees to certain special values of Artin-Ihara $L$-functions, which can be computed explicitly in this setting. In \cite{dubose-vallieres} the above approach was generalised to $\Z_p^l$-covers of graphs. 

In \cite{gonet} Gonet proved an analogous growth formula for $\Z_p$-covers of graphs via algebraic means. To be more precise, Gonet introduced suitable Iwasawa modules and studied their structure as $\Lambda$-modules in order to derive the growth formula. Therefore Gonet's proof was very much in the flavour of Iwasawa's classical approach. In \cite{Kleine-Mueller4} we generalised Gonet's approach and studied more generally $\Z_p^l$-covers. One of the most important properties of the corresponding Iwasawa modules is that they do not contain any non-trivial \emph{pseudo-null} $\Lambda$-submodules (see Section~\ref{section:notation} for the precise definition). The absence of pseudo-null submodules did not only simplify the investigation of the asymptotic growth of the number of spanning trees, but it also allowed us to formulate and prove an Iwasawa main conjecture for $\Z_p^l$-covers of graphs in a rather straight-forward way (see \cite[Theorem~5.2]{Kleine-Mueller4}). 

The main goal of the present article is to derive Iwasawa main conjectures in a much more general setting. We study voltage covers $X_\infty/X$ such that the corresponding voltage assignment takes values in an arbitrary uniform pro-$p$ group $G$ (for the detailed definitions we refer to Section~\ref{section:notation}). Note that any compact $p$-adic Lie group contains an open normal subgroup which is uniform. Therefore this class of voltage covers is fairly general. 

To any finite Galois cover $Y/X$ and a $\Gal(Y/X)$-representation $\rho$ one can attach a complex-analytic Artin-Ihara $L$-function $L(\rho, Y/X, u)$. For some exponent $d$ one can write $L(\rho, Y/X, u)^{-1} = (1-u^2)^d \cdot h(\rho, X/Y,u)$, where we understand the function $h(\rho, Y/X, u)$, which does not vanish at ${u = 1}$, as the \lq algebraic part' of the $L$-function. Now suppose that $Y$ is the derived graph of a voltage assignment $\alpha$. We will prove in Section~\ref{section:Ihara-L-functions} that the special value at ${u = 1}$ of this algebraic part satisfies 
\begin{align} \label{eq:algebraicpart} h(\rho, Y/X, 1) = \det(\rho(D - A_\alpha^t)), \end{align} 
where $D$ and $A_\alpha^t$ denote the degree matrix and the transpose of the adjacency matrix attached to $Y/X$ (see Corollary~\ref{cor:interpolation}). As $\Z_p\llbracket G\rrbracket$ is non-abelian we cannot apply the determinant to $D-A_\alpha^t$ directly. But we can define a reduced norm on $\Z_p\llbracket G\rrbracket$ with the property that \begin{equation}
\label{eq:intro:det}
    \det(\rho(D-A_\alpha^t))I_d=\rho(\textup{Nrd}(D-A_\alpha^t)), 
\end{equation} where $I_d$ is the ${d \times d}$ identity matrix and $d=\dim(\rho)$ (see Lemma~\ref{lemma-interpolation}). In our formulation of the main conjecture $\textup{Nrd}(D-A_\alpha^t)$ will play the role of the $p$-adic $L$-function. 

Now let $X_\infty$ be the derived graph of a voltage cover $\alpha$ which takes values in a uniform pro-$p$ group $G$. Let $X_n$, ${n \in \N}$, be the intermediate finite voltage covers of $X$ such that ${\Gal(X_n/X) \cong G/G^{p^n}}$ for each $n$. Let $\textup{Pic}_\infty$ be the projective limit of the Picard groups $\textup{Pic}(X_n)$ (see also Lemma~\ref{tensor-proj-lim}). Then $\textup{Pic}_\infty$ is a finitely generated and torsion $\Lambda$-module, where ${\Lambda = \Z_p\llbracket G\rrbracket }$ denotes the Iwasawa algebra. In classical commutative Iwasawa theory a main conjecture typically describes a generator of the characteristic ideal of an Iwasawa module in terms of special values of $p$-adic analytic $L$-functions. Since $\Lambda$ is in general non-commutative in the present paper, we instead use the language of \emph{Fitting invariants}, which we recall in Section~\ref{section:Fitting}. The notion of Fitting invariants in the sense we are using here was introduced by Nickel (see \cite{nickel-definition} and \cite{nickel}) and was first applied to Iwasawa modules coming from ideal class groups in non-abelian CM extensions. The main conjecture which we prove in Section~\ref{section:main_conjecture} then takes the following form. 
\begin{thm} \label{thm:main_conjecture} 
   In the above setting, suppose that $X_n$ is connected for each ${n \in \N}$. Then 
   \[ \textup{Fitt}_\Lambda(\textup{Pic}_\infty) = \langle \textup{Nrd}(D - A_\alpha^t)\rangle_{I(\Lambda)}, \]
   where the equality is one of $I(\Lambda)$-ideals. 
\end{thm} 
Here the left hand side is the Fitting invariant of the Iwasawa module $\textup{Pic}_\infty$, $\textup{Nrd}$ denotes the \emph{reduced norm} (introduced in Sections~\ref{section:Ihara-L-functions} and \ref{section:Fitting}), and $I(\Lambda)$ is a certain local ring which we define in Section~\ref{section:Fitting}. The proof uses once again equation~\ref{eq:intro:det}. In view of equation~\eqref{eq:algebraicpart} this provides us with a link to the algebraic part of the Ihara $L$-function. Moreover, the Fitting invariant of $\textup{Pic}_\infty$ also contains information on the global annihilators of the $\Lambda$-module $\textup{Pic}_\infty$. 

At the end of Section~\ref{section:main_conjecture} we also prove one inclusion of a main conjecture for the Iwasawa module $J_\infty$ which is the projective limit of the \emph{Jacobians} of the $X_n$. Both main conjectures are closely related; in fact, in the abelian case we are able to derive a full main conjecture for the Jacobians. The main reason for this is that in the abelian case we formulate the Iwasawa main conjecture in terms of characteristic ideals. If $G$ is abelian and the dimension of $G$ is at least $2$, then the two $\Lambda$-modules $J_\infty$ and $\textup{Pic}_\infty$ have the same characteristic ideal. 

Iwasawa main conjectures in non-commutative Iwasawa theory for number fields are usually formulated in the language of algebraic $K$-theory (see e.g. \cite{ritter-weiss, kakde}). In order to make this work, one restricts (see \cite[Section~3]{5people}) to those Iwasawa-modules which satisfy a certain \emph{$\mathfrak{M}_H(G)$-property} (see Definition~\ref{def:inv}(v)). We reformulate Theorem~\ref{thm:main_conjecture} in the language of algebraic $K$-theory in Section~\ref{section:K-theory}. The $K$-theoretic formulation is slightly stronger as the one given in Theorem~\ref{thm:main_conjecture} but we can only prove it under the assumption that the $\mathfrak{M}_H(G)$-property holds.  We study the $\mathfrak{M}_H(G)$-property of $\Lambda$-modules in some detail in the last section, proving a sufficient condition for the validity of this property and working out some concrete numerical examples. 

\vspace{2mm} 
Let us finally summarise the \emph{structure of the article}. In Section~\ref{section:notation} we introduce the necessary background from both the theory of Iwasawa modules in a non-commutative setting on the one hand and the theory of graphs on the other hand (in particular we briefly describe the language of voltage covers). The next Section~\ref{section:algebraic_theory} contains results on the algebraic structure of the Iwasawa modules which are studied in this paper. Most importantly, we prove that these Iwasawa modules do not contain any non-trivial pseudo-null submodules (see Theorem~\ref{thm:no-pseudo-null}), a crucial fact which will simplify a lot of the arguments in the sequel. 

In Section~\ref{section:Ihara-L-functions} we recall the Artin-Ihara $L$-functions, and we describe special values of the algebraic part of these $L$-functions in terms of determinants and reduced norms. Section~\ref{section:Fitting} contains a collection of definitions and basic results on Fitting invariants and some auxiliary results on properties of these invariants which will be needed for the formulation and proof of a main conjecture in Section~\ref{section:main_conjecture}. In Section~\ref{section:K-theory} we elaborate on a reformulation of the main conjecture in the terms of algebraic $K$-theory. The final section contains a proof of a sufficient criterion for the validity of the $\mathfrak{M}_H(G)$-property. We end by describing several concrete numerical examples to which the criterion can be applied. 

\textbf{Acknowledgements.} The authors would like to thank Antonio Lei and Andreas Nickel for helpful remarks on an earlier version of this article. 

\section{Background} \label{section:notation} 
\subsection{Uniform groups, group rings and Iwasawa modules} 
Let $p$ be a fixed rational prime. By a well-known result of Lazard (see \cite[Corollary~8.34]{Dixon}) any compact $p$-adic Lie group contains an open normal subgroup which is \emph{uniform}. Therefore we will restrict to uniform groups $G$ in this article. 

A profinite pro-$p$-group $G$ which is topologically generated by $l$ elements is called \emph{uniform of dimension $l$} if we have a filtration 
\[ G \supseteq G^p \supseteq G^{p^2} \supseteq \cdots \supseteq G^{p^n} \supseteq \cdots \]
such that $G^{p^{n+1}} \subseteq G^{p^n}$ is normal and ${G^{p^n}/G^{p^{n+1}} \cong (\Z/p\Z)^l}$ for each ${n \in \N}$ (see \cite[Definition~1.15, Lemma~2.4 and Theorem~3.6]{Dixon}). 

Although in general not commutative, uniform $p$-groups share many nice properties. For example, any uniform $p$-group $G$ is torsion-free by \cite[Theorem~4.5]{Dixon}, and ${G_n := G^{p^n}}$ are exactly the groups in the lower central series of $G$. In the following, we let ${G^{(n)} = G/G_n}$ for every ${n \in \N}$ (this makes sense since ${G_n \subseteq G}$ are normal subgroups). 

Moreover, the cohomological dimension of $G$ is equal to its dimension $l$ by results of Lazard and Serre (see \cite{lazard} and \cite[Corollaire~1]{serre}). 

Now we define several group rings, which will play a major role in our article. To this purpose, fix a uniform $p$-group $G$ of dimension $l$, and let ${R = \Z_p[G]}$ and ${\Lambda = \Z_p\llbracket G\rrbracket }$. For example, if ${G \cong \Z_p^l}$, then the completed group ring $\Lambda$ is non-canonically isomorphic to the ring $\Z_p\llbracket T_1, \ldots, T_l\rrbracket $ of power series in $l$ variables over $\Z_p$ (the isomorphism depends on the choice of a set of topological generators of $\Z_p^l$). For any ${n \in \N}$ we can also consider the group ring $\Z_p[G^{(n)}]$ of the finite quotient group $G^{(n)}$, and we have natural surjective maps 
\[ R \longrightarrow \Z_p[G^{(n)}], \quad \Lambda \longrightarrow \Z_p[G^{(n)}]. \] 
Let $I_n$ be the kernel of the latter map, ${n \in \N}$. 

Since $G$ is a pro-$p$-group without $p$-torsion, the completed group ring $\Lambda$ is a Noetherian ring without zero divisors by \cite[Corollary~7.25]{Dixon}. Moreover, $\Lambda$ is Auslander regular by \cite[Theorem~3.26]{venjakob}, and it is a local ring with finite global dimension by \cite{brumer}. In this article, we will call $\Lambda$ the \emph{Iwasawa algebra} of $G$, and by an \emph{Iwasawa module} we will mean any finitely generated $\Lambda$-module (here and in what follows, we always mean \emph{left} $\Lambda$-modules). In the following, we describe the basic properties and invariants of such an Iwasawa module $M$ which we are interested in. 

\begin{def1} \label{def:inv} Let $M$ be an Iwasawa module. \begin{compactenum}[(i)] 
  \item Letting $\mathcal{Q}(G)$ be the skew field of fractions of $\Lambda$ (which exists by \cite[Corollary~7.25]{Dixon} and \cite[Chapter~10]{goodearl-warfield}), we define the $\Lambda$-rank of $M$ to be 
\[ \textup{rank}_\Lambda(M) = \dim_{\mathcal{Q}(G)}(\mathcal{Q}(G) \otimes_\Lambda M). \] 
\item An Iwasawa module is called \emph{torsion} if ${\textup{rank}_\Lambda(M) = 0}$, and a torsion Iwasawa module is called \emph{pseudo-null} if ${\textup{Ext}^1_{\Lambda}(M, \Lambda) = 0}$. 
\item Let $N$ be another Iwasawa module. Then $M$ is called \emph{pseudo-isomorphic to} $N$ if there exists a \emph{pseudo-isomorphism} ${\varphi \colon M \longrightarrow N}$, i.e. a $\Lambda$-module homomorphism with pseudo-null kernel and cokernel. In this case we write ${M \sim N}$. 
\item Let $M[p^\infty]$ denote the $p$-torsion submodule of $M$, and let $\F_p$ be the finite field with $p$ elements. Following Howson (see \cite[equation~33)]{howson}), we define the \emph{$\mu$-invariant} of $M$ as 
\[ \mu_\Lambda(M) = \sum_{i \ge 0} \rg_{\F_p\llbracket G\rrbracket }(p^i M[p^\infty]/p^{i+1}). \] 
Note that this sum is finite since $\Lambda$ is Noetherian and therefore ${M[p^\infty] \subseteq M}$ is a finitely generated $\Lambda$-module, i.e. ${p^i M[p^\infty] = 0}$ for sufficiently large $i$. 
\item Suppose that ${H \subseteq G}$ denotes a normal subgroup such that ${G/H \cong \Z_p}$. Then we say that $M$ \emph{satisfies the $\mathfrak{M}_H(G)$-property} if $M/M[p^\infty]$ is finitely generated as a module over the group ring $\Z_p\llbracket H\rrbracket $. 
\end{compactenum} 
\end{def1} 
It follows from \cite[Proposition~3.34]{venjakob} that ${M \sim N}$ for two Iwasawa modules $M$ and $N$ implies that ${\mu_\Lambda(M) = \mu_\Lambda(N)}$. It is obvious that also the ranks of $M$ and $N$ are equal in this case. 
\begin{rem} 
  The last property~(v) from Definition~\ref{def:inv}, however, is \emph{not} invariant under pseudo-isomorphisms. For example, consider the case ${G \cong \Z_p^3}$. We identify $$\Lambda = \Z_p\llbracket T_1, T_2, T_3\rrbracket , \quad \Z_p\llbracket H\rrbracket = \Z_p\llbracket T_1, T_2\rrbracket .$$ 
  The $\Lambda$-module ${M = 0}$ obviously satisfies the $\mathfrak{M}_H(G)$-property. Since ${N = \Lambda/(T_1, T_2)}$ is pseudo-null over $\Z_p\llbracket T_1, T_2, T_3\rrbracket $, the natural inclusion ${M \longrightarrow N}$ is a pseudo-isomorphism. But ${N/N[p^\infty] \cong N}$ is not finitely generated as a $\Z_p\llbracket T_1, T_2\rrbracket $-module. 
\end{rem} 
Let us make one more remark concerning the last part of the above definition. It has been shown by Kundu and Lim (see \cite[Lemma~3.1]{kundu-lim}) that in order to have ${G/H \cong \Z_p}$, $H$ must be a uniform group, too. Therefore $\Z_p\llbracket H\rrbracket $ is a Noetherian ring without zero divisors. 

In the next two lemmas we collect useful facts concerning $\mu$-invariants and the $\mathfrak{M}_H(G)$-property. 
\begin{lemma} \label{lemma:mu} 
  Let $G$ be a compact pro-$p$, $p$-adic Lie group such that ${G[p^\infty]= \{0\}}$. Then the following statements hold. \begin{compactenum}[(a)] 
  \item For every Iwasawa module $M$ we have 
  \[ \mu_\Lambda(M) = \sum_{i \ge 0} (-1)^i \textup{ord}_p(H_i(G, M[p^\infty])). \] 
  \item Let $H$ be a closed normal subgroup of $G$ such that ${\Gamma := G/H}$ is isomorphic to $\Z_p$. Then for every Iwasawa module $M$ we have 
  \[ \mu_\Lambda(M) = \sum_{i \ge 0} (-1)^i \mu_\Gamma(H_i(H, M[p^\infty])). \]
  \item Let ${H \subseteq G}$ be a closed subgroup such that ${G/H \cong \Z_p}$. Then any Iwasawa module $M$ which is finitely generated as a $\Z_p\llbracket H\rrbracket $-module has $\mu_\Lambda$-invariant zero as a $\Lambda$-module. 
  \end{compactenum} 
\end{lemma} 
\begin{proof} 
  Parts (a) and (c) are proved in \cite[Corollary~1.7]{howson} and \cite[Lemma~2.7]{howson}. Assertion (b) follows from (a) and \cite[Lemma~2.2]{Lim_MHG}. 
\end{proof} 
\begin{lemma} \label{lemma:MHG} 
   Let 
  \[ 0 \longrightarrow A \longrightarrow B \stackrel{\alpha}{\longrightarrow} C \longrightarrow 0 \]
   be an exact sequence of Iwasawa modules. Then $B$ has the $\mathfrak{M}_H(G)$-property if and only if both $A$ and $C$ have the $\mathfrak{M}_H(G)$-property. In particular, the following statements hold.  \begin{compactenum}[(a)] 
   \item Suppose that $A$ has the $\mathfrak{M}_H(G)$-property. Then $B$ satisfies the $\mathfrak{M}_H(G)$-property if and only if $C$ satisfies the $\mathfrak{M}_H(G)$-property. 
   \item Suppose that $C$ has the $\mathfrak{M}_H(G)$-property. Then $A$ satisfies the $\mathfrak{M}_H(G)$-property if and only if $B$ satisfies the $\mathfrak{M}_H(G)$-property. 
   \end{compactenum} 
\end{lemma} 
\begin{proof} 
  It suffices to show that $B$ has the $\mathfrak{M}_H(G)$-property if and only if $A$ and $C$ have this property. To this purpose, choose $m$ large enough such that ${A[p^m] = A[p^\infty]}$, ${B[p^m] = B[p^\infty]}$ and ${C[p^m] = C[p^\infty]}$. Then we have an exact sequence 
  \[ 0 \longrightarrow A[p^\infty] \longrightarrow B[p^\infty] \stackrel{\alpha}{\longrightarrow} C[p^m] \stackrel{\beta}{\longrightarrow} A_f \longrightarrow B_f \longrightarrow C_f \longrightarrow 0. \] 
  Here we denote, by abuse of notation, the restriction of $\alpha$ to ${B[p^\infty] \subseteq B}$ again by $\alpha$. In the following we adopt an argument from the proof of \cite[Lemma~3.3]{Lim_MHG2}. In view of the commutative diagram 
  \[ \xymatrix{& 0 \ar[d] & 0 \ar[d] & 0 \ar[d] & \\ 
  0 \ar[r] & A[p^\infty] \ar[d] \ar[r] & B[p^\infty] \ar[d] \ar[r] & \textup{Im}(\alpha) \ar[d] \ar[r] & 0 \\ 
  0 \ar[r] & A \ar[r] & B \ar^\alpha[r] & C \ar[r] & 0}\] 
  the snake lemma yields an exact sequence 
  \begin{align} \label{eq:MHG}  0 \longrightarrow A_f \longrightarrow B_f \longrightarrow C/\textup{Im}(\alpha) \longrightarrow 0. \end{align} 

  Now consider the following commutative diagram. 
  \[ \xymatrix{& 0 \ar[d] & 0 \ar[d] & \textup{Im}(\beta) \ar[d] &  \\ 
      0 \ar[r] & \textup{Im}(\alpha) \ar[r] \ar[d] & C[p^\infty] \ar[d] \ar[r] & \textup{Im}(\beta) \ar[d] \ar[r] & 0 \\ 
      0 \ar[r] & C \ar[r]^{\textup{id}} & C \ar[r] & 0 & }\]
Then the snake lemma yields an exact sequence 
\begin{align} \label{eq:MHG2} 0 \longrightarrow \textup{Im}(\beta) \longrightarrow C/\textup{Im}(\alpha) \longrightarrow C_f \longrightarrow 0.\end{align}  

Suppose first that $B$ has the $\mathfrak{M}_H(G)$-property. Then it follows from the exact sequence~\eqref{eq:MHG} that $A_f$ is finitely generated over $\Z_p\llbracket H\rrbracket $, i.e. $A$ has the $\mathfrak{M}_H(G)$-property. Moreover, this exact sequence also implies that $C/\textup{Im}(\alpha)$ is finitely generated over $\Z_p\llbracket H\rrbracket $, and therefore the exact sequence~\eqref{eq:MHG2} implies that $C$ has the $\mathfrak{M}_H(G)$-property. 

On the other hand, if $A$ has the $\mathfrak{M}_H(G)$-property then ${\textup{Im}(\beta) \subseteq A_f}$ is a finitely generated $\Z_p\llbracket H\rrbracket $-module, and it follows from the exact sequence~\eqref{eq:MHG2} that (provided that also $C$ has the $\mathfrak{M}_H(G)$-property) also $C/\textup{Im}(\alpha)$ is finitely generated over $\Z_p\llbracket H\rrbracket $. Since $A$ has the $\mathfrak{M}_H(G)$-property, the exact sequence~\eqref{eq:MHG} implies that $B$ has also the $\mathfrak{M}_H(G)$-property. 
\end{proof} 
The investigation of the $\mathfrak{M}_H(G)$-property has been introduced by Coates-Sujatha and it is very important in the context of Iwasawa main conjectures in non-\-commu\-ta\-tive Iwasawa theory (see Section~\ref{section:K-theory} below). 

We conclude with a variant of \emph{Nakayama's Lemma}, which will be useful for our purposes. 
\begin{lemma}\label{nakayama} Let $\mathfrak{m}$ be the maximal ideal of $\Lambda$ and let $N$ be a finitely generated $\Lambda$-module. If $x_1,\dots, x_n$ generate $N/\mathfrak{m}N$ as a $\Lambda/\mathfrak{m}$-module, then $x_1,\dots ,x_n$ generate $N$ as a $\Lambda$-module.
\end{lemma}
\begin{proof}
This follows from \cite[Lemma~1.1]{venjakob} since $G$ is a pro-$p$ group.     
\end{proof}

\subsection{Graph theoretical definitions}
Let $X$ be a connected undirected graph. We denote the set of vertices of $X$ by $V(X)$ and the set of edges by $E(X)$. We write $\mathbf{E}(X)$ for the set of directed edges. For each ${v,w \in V(X)}$ we denote by $E(v,w)$ the set of (undirected) edges between $v$ and $w$. In this article, we assume that $E(v,w)$ is finite for each choice of ${v, w \in V(X)}$, and we denote by 
\[ \deg(v) = \sum_{w \in V(X)} |E(v,w)| \]
the degree of $v$. Moreover, we let ${\textup{mult}(v,w) = |E(v,w)|}$. 

We define $\textup{Div}(X)$ to be the free abelian group on $V(X)$, i.e. 
\[ \textup{Div}(X) = \left\{ \sum_{v \in V(X)} a_v v \mid \textup{$a_v \in \Z_p$ for each $v$, $a_v = 0$ for all but finitely many $v$}\right\}, \]
and we let $\textup{Div}^0(X)$ be the subgroup of $\textup{Div}(X)$ defined as 
\[ \textup{Div}^0(X) = \left\{ \sum_{v \in V(X)} a_v v \in \textup{Div}(X) \; \Big| \; \sum_v a_v = 0 \right\}. \] 
Finally, we define the group of principal divisors by 
\[ \textup{Pr}(X) = \{ d \in \textup{Div}^0(X) \mid d \simeq 0 \}, \] 
where $d \simeq 0$ means that there exists a sequence of \emph{firing moves} that transforms $d$ into 0. This means that there exists a finite sequence ${v_1, \ldots, v_s}$ of vertices ${v_i \in V(X)}$ and a sequence of integers ${a_i \in \{\pm 1\}}$ such that 
\[ d = \sum_{i = 1}^s (-1)^{a_i} \left( \sum_{w: (v_i,w) \in E(X)} (v_i - w) \right)\] 
(see also \cite[Definition~1.5]{corry-Perkinson}). 
\begin{def1} 
  We let $\textup{Pic}(X) = \textup{Div}(X)/\textup{Pr}(X)$ and $J(X) = \textup{Div}^0(X)/\textup{Pr}(X)$. 
\end{def1} 
The Picard groups and the Jacobians of suitable families of graphs will be the main objects of interest in this article. 

Now fix a finite graph $X$ and a uniform $p$-group $G$, and consider a sequence of Galois covers 
\[X\longleftarrow X_1 \longleftarrow X_2\longleftarrow \dots\]
of $X$ such that ${\Gal(X_n/X)\cong G^{(n)} = G/G^{p^n}}$ for each ${n \in \N}$ (for a good introduction to the theory of Galois covers we refer to \cite[Chapter~5]{sunada}). 

\emph{In all what follows, we always assume that each intermediate graph $X_n$ is connected.}

We will prove a criterion for this condition for the special class of graphs on which this article focuses at the end of this section (see Lemma~\ref{connected} below). 

An important class of Galois covers of graphs arises via so-called \emph{voltage covers} (for example, any Galois cover $Y/X$ of connected graphs with abelian Galois group arises from such a voltage assignment, see \cite[Section~3]{dubose-vallieres}). Let $S$ be the image of a section of the natural surjective map $E(X)\longrightarrow \mathbf{E}(X)$. Let $\alpha$ be a map (a so-called \emph{voltage assignment}) 
\[\alpha \colon S\longrightarrow G\] and let
\[\alpha_n\colon S\longrightarrow G^{(n)}\] 
be the induced maps. We extend $\alpha$ and $\alpha_n$ naturally to the whole set $\mathbf{E}_X$. We assume that $X_n$ is the derived graph of the voltage assignment $\alpha_n$, by which we mean the following. The vertices of $X_n$ are denoted as tuples $(v,g)$, where $g\in G^{(n)}$ and $v\in V(X)$. There is an edge between $(v,g)$ and $(v',g')$ if and only if there is an edge $e$ from $v$ to $v'$ in $\mathbf{E}_X$ such that $g\alpha_n(e)=g'$. We can embed $V(X)$ into each $V(X_n)$ by writing ${v_i = (v_i, 1)}$ for each $i$. We have a natural action of $G^{(n)}$ on $X_n$ given by
$g(v, g')=(v, gg')$. This induces a well-defined action of $G^{(n)}$ on $\textup{Div}(X_n)$, $J(X_n)$, etc. Let $X_\infty$ be the derived graph associated to $\alpha$. 
Let $\pi_n \colon \textup{Div}(X_\infty)\longrightarrow \textup{Div}(X_n)$ be the natural map. By abuse of notation we also denote the natural projection $\Lambda\longrightarrow \Z_p[G^{(n)}]$ by $\pi_n$. 

Finally, we define the principal divisors 
\[ p_{i,0} := p_i = \deg(v_i) v_0 - \sum_{w : v_i \sim w} \sum_{e_j \in E(v_i,w)} \alpha(e_j) w, \] 
where $v_i \simeq w$ if and only if there exists a finite sequence of firing moves that transforms $v_i$ into $w$ (as in the definition of $\textup{Pr}(X)$). 
If $A$ is an $R$-module we denote by $A_\Lambda$ the $\Lambda$-module $N\otimes_R\Lambda$. For brevity we will write
\[ \textup{Div}_\Lambda := \textup{Div}(X_\infty) \otimes_R \Lambda, \quad \textup{Div}_\Lambda^0 := \textup{Div}^0(X_\infty) \otimes_R \Lambda, \quad \textup{Pr}_\Lambda := \textup{Pr}(X_\infty) \otimes_R \Lambda. \]
We let $M_\Lambda$ be the submodule of $\textup{Div}_\Lambda$ which is generated by the elements ${p_i}$, ${1 \le i \le m}$ (recall ${V(X) = \{v_1, \ldots, v_m\}}$), the elements ${v_i - v_j}$ for ${1 \le j < i \le m}$, and ${I_0 \cdot \textup{Div}_\Lambda}$, where ${I_0 = \ker(\pi_0)}$, as above. 

For a voltage cover ${X_\infty}$ attached to a voltage assignment ${\alpha \colon S \longrightarrow G}$, we can prove a sufficient condition for the $X_n$ to be connected. 
\begin{def1}
Let $c$ be a path in $X$ consisting of the edges $e_1,e_2,\dots, e_s$. We define $\beta(c)=\prod_{i=1}^s\alpha(e_i)$. We call a closed path without backtracks a cycle. 
\end{def1}
The following Lemma is basically a reformulation of one inclusion of \cite[Theorem 4]{diss-gonet}. We nevertheless give a full proof for the convenience of the reader.
\begin{lemma}\label{connected}
Let $G$ be a uniform $p$-adic Lie group of dimension $l$. Assume that $X$ is connected, and that there is a vertex $v_0$ in $X$ and cycles $C_1,\dots,C_l$ passing through $v_0$ such that $\{\beta(C_1),\dots, \beta(C_l)\}$ is a set of topological generators of $G$. Then $X_n$ is connected for all $n$.
\end{lemma}
\begin{proof}
For two vertices $(v,g)$ and $(v', g')$ in $V(X_n)$ we write ${(v,g) \sim (v', g')}$ if there exists a path from $(v,g)$ to $(v', g')$. This defines an equivalence relation on the set of vertices $V(X_n)$, and we want to show that there is only one equivalence class. 

Each element in $G^{(n)}$ can be written as a linear combination of $\beta_i=\beta(C_i)$ (by abuse of notation we also view $\beta(C_i)$ as generators of $G^{(n)}$). By definition of the derived graph $X_n$, $(v_0,g)$ is connected to $(v_0,g\beta_i)$ for $1\le i\le l$. Let now $v$ be an arbitrary vertex of $X$. Fix a path $c$ from $v$ to $v_0$ and let $\gamma=\beta(c)$. Then 
\[(v,g)\sim (v_0,g\gamma)\sim (v_0,g\gamma\beta_i)\sim (v,g\gamma\beta_i\gamma^{-1}).\]
As the $\beta_i$ are topological generators of $G$ we see that $(v,g)$ is connected to every $(v,g')$. This suffices for proving the lemma since $X$ is connected by assumption. 
\end{proof}

\section{algebraic structure theorems} \label{section:algebraic_theory} 
The aim of the section is to describe the algebraic structure of the modules $\textup{Div}_\Lambda$, $\textup{Pr}_\Lambda$ and $\textup{Div}^0_\Lambda$ as $\Lambda$-modules. Ultimately we are interested in quotients of these modules and we will show that $\textup{Div}_\Lambda/\textup{Pr}_\Lambda$ actually coincides with the projective limit of the groups $\textup{Pr}(X_n)$.
\begin{lemma}
\label{lem:free}
    $\textup{Pr}_\Lambda$ is $\Lambda$-free and embeds naturally into $\textup{Div}_\Lambda$.
\end{lemma}
\begin{proof}
    Note that $\textup{Div}_\Lambda$ is $\Lambda$-free by definition. $\textup{Pr}_\Lambda$ is generated by the elements $p_{i,0}$ as a $\Lambda$-module. If we can show that these generators do not satisfy a non-trivial $\Lambda$-relation, then we are done. For each of the finite groups $G^{(n)}$ we let ${N_n=\sum_{g\in G^{(n)}}g}$. Then $N_n\Z_p[G^{(n)}]\cong \Z_p$ and it is easy to verify that ${\bigcap_{n \in \N} \pi_n^{-1}(N_n\Z_p[G^{(n)}])=\{0\}}$. Assume now that we have a vanishing linear combination $\sum F_ip_{i,0}$. We can assume that $F_0\neq 0$ and we can choose $n$ large enough such that ${F_0\notin \pi_n^{-1}(N_n\Z_p[G^{(n)}])}$.
    Note that $\textup{Pr}(X_n)$ has $\Z_p$-rank $\vert G^{(n)}\vert \cdot \vert X\vert-1$ and that the only (up to scalar multiples) non-trivial vanishing linear combination of the $p_{i,0}$ in $\textup{Div}(X_n)$ is given by 
    \[\sum N_np_{i,0} = 0. \]
    Indeed, if there were more such linear combinations then the $\Z_p$-rank of the module generated by the $p_{i,0}$ would have rank smaller than $\vert G^{(n)}\vert \cdot \vert X\vert -1$.
    On the other hand we also have
    \[\sum \pi_n(F_i)p_{i,0}=0,\]
    yielding $\pi_n(F_i)\in N_n\Z_p[G^{(n)}]$ for each $i$, which is impossible for ${i = 0}$.  Thus, such a linear combination does not exist and $\textup{Pr}_\Lambda$ is indeed $\Lambda$-free and embeds naturally into $\textup{Div}_\Lambda$.
\end{proof}
\begin{lemma} \label{lemma:surjective} 
    Each $\pi_n$ induces surjective maps 
    \[ \textup{Div}^0(X_\infty)\longrightarrow \textup{Div}^0(X_n), \quad \textup{Pr}(X_\infty)\longrightarrow \textup{Pr}(X_n). \]  
\end{lemma}
\begin{proof}
    Note that $\textup{Div}^0(X_n)$ is a finitely generated $\Z_p[G^{(n)}]$-module. Recall that $I_n$ denotes the kernel of the natural map $\Lambda\longrightarrow \Z_p[G^{(n)}]$. Let $\overline{I}_0$ be the image of $I_0$ under the natural projection $\Lambda\longrightarrow \Z_p[G^{(n)}]$. Then $\textup{Div}^0(X_n)$ is as a $\Z_p[G^{(n)}]$-module generated by $\overline{I}_0\textup{Div}(X_n)$ and the elements $v_i-v_j$ for $v_i,v_j\in V(X)$. As all the generators have pre-images in $\textup{Div}^0(X_\infty)$ the first claim follows. 
    The proof for $\textup{Pr}(X_n)$ is similar. We just have to work with the generators $p_{i,0}$ instead. 
\end{proof}
\begin{lemma} \label{lemma:kernel} 
    The kernel of $\pi_n \colon \textup{Div}(X_\infty) \longrightarrow \textup{Div}(X_n)$ is $I_n\textup{Div}_\Lambda\cap \textup{Div}(X_\infty)$.
\end{lemma}
\begin{proof}
    The group of divisors $\textup{Div}(X_n)$ is a free $\Z_p[G^{(n)}]$-module in the generators $v_0,v_1,\dots, v_m$, where ${|V(X)| = m}$. Thus, we obtain isomorphisms
    \begin{eqnarray*} \textup{Div}(X_n) & \cong & 
    \bigoplus_{i=0}^m\Z_p[G^{(n)}]v_i \; \cong \; \bigoplus_{i=0}^m(\Lambda/I_n) v_i \\ 
    & \cong & \textup{Div}_\Lambda/I_n \; \cong \; \textup{Div}(X_\infty)/(I_n\textup{Div}_\Lambda\cap \textup{Div}(X_\infty)).\end{eqnarray*} 
\end{proof}

Recall that $M_\Lambda$ denotes the submodule of $\textup{Div}_\Lambda$ which is generated by the $p_{i,0}$, the elements ${v_{i} - v_{j}}$ for ${1 \le j < i \le m}$ and $I_0 \textup{Div}_\Lambda$. 
\begin{lemma} \label{lemma:M_Lambda} We have
    \[I_n\textup{Div}_\Lambda +\textup{Div}^0(X_\infty)=M_\Lambda\]
    and \[\textup{Pr}_\Lambda+I_n\textup{Div}_\Lambda=I_n\textup{Div}_\Lambda+\textup{Pr}(X_\infty).\] 
\end{lemma}
\begin{proof}
    We only give a proof for the first claim -- the second one can be proved analogously. Note that there is a natural projection 
    \[ M_\Lambda \longrightarrow M_\Lambda/I_n \longrightarrow\textup{Div}^0(X_n).\] 
    Likewise, ${\textup{Div}^0(X_\infty)\longrightarrow \textup{Div}^0(X_n)}$ is surjective by Lemma \ref{lemma:surjective}. Let ${y\in M_\Lambda}$ and let $\overline{y}$ be its image in $\textup{Div}^0(X_n)$. Let ${z\in \textup{Div}^0(X_\infty)}$ be a preimage of $\overline{y}$ under $\pi_n$. It follows that $z-y\in \ker (\textup{Div}_\Lambda\longrightarrow \textup{Div}(X_n))=I_n\textup{Div}_\Lambda$. Therefore, 
    \[y\in z+I_n\textup{Div}_\Lambda.\] 
    It follows that $M_\Lambda\subseteq \textup{Div}^0(X_\infty)+I_n\textup{Div}_\Lambda$. As the other inclusion is obvious the claim follows. 
\end{proof}
\begin{lemma} \label{lemma:J} We have an isomorphism
    \[J(X_n) \cong M_\Lambda/(\textup{Pr}_\Lambda +I_n\textup{Div}_\Lambda).\]
\end{lemma}
\begin{proof}
    By Lemmas~\ref{lemma:kernel} and \ref{lemma:M_Lambda} we have \[\textup{Div}^0(X_n)\cong \textup{Div}^0(X_\infty)/(I_n\textup{Div}_\Lambda\cap \textup{Div}^0(X_\infty)) \cong M_\Lambda/I_n\textup{Div}_\Lambda. \]
    Likewise we have
    \begin{eqnarray*} 
    \textup{Pr}(X_n) & \cong & \textup{Pr}(X_\infty)/(I_n \textup{Div}_\Lambda \cap \textup{Pr}(X_\infty)) \\ & \cong & (\textup{Pr}(X_\infty) + I_n \textup{Div}_\Lambda)/I_n \textup{Div}_\Lambda \\ 
    & = & (\textup{Pr}_\Lambda +I_n\textup{Div}_\Lambda)/I_n\textup{Div}_\Lambda. \end{eqnarray*} 
    Taking the quotient gives the claim.
\end{proof}

\begin{lemma} \label{lemma:iso} We have an isomorphism
    \[\textup{Div}^0_\Lambda\cong M_\Lambda. \]
\end{lemma}
\begin{proof} Note that the generators of $\textup{Div}^0(X_\infty)$ as a $R$-module all lie in $M_\Lambda$.
   Since $M_\Lambda$ can be identified with the image of the natural  map ${\textup{Div}^0_\Lambda\longrightarrow \textup{Div}_\Lambda}$, we have a surjective map $\textup{Div}^0_\Lambda \longrightarrow M_\Lambda$. It remains to show that the natural map ${\textup{Div}^0_\Lambda\longrightarrow \textup{Div}_\Lambda}$ is indeed injective. 
   Since $\textup{Div}(X_\infty)/\textup{Div}^0(X_\infty)$ is annihilated by $I_0$, the image of 
   \[ \textup{Tor}^1(\textup{Div}(X_\infty)/\textup{Div}^0(X_\infty),\Lambda)\] 
   in $\textup{Div}^0_\Lambda$ is annihilated by $I_0$. On the other hand ${\textup{Div}^0_\Lambda = \textup{Div}^0(X_\infty)\otimes \Lambda}$ is $\Lambda$-torsion free. Thus, the natural map $\textup{Div}^0_\Lambda\longrightarrow \textup{Div}_\Lambda$ is indeed injective. 
\end{proof}
\begin{def1}
 We let $J_\infty = J(X_\infty) \otimes_R \Lambda$ and $\textup{Pic}_\infty=\textup{Pic}(X_\infty)\otimes_R\Lambda$.
\end{def1}
The above lemma has the following important consequence. 
\begin{corollary} \label{cor:J_infty} 
    We have short exact sequences
    \[0\longrightarrow \textup{Pr}_\Lambda \longrightarrow \textup{Div}^0_\Lambda\longrightarrow J_\infty\longrightarrow 0 \]and
    \[0\longrightarrow \textup{Pr}_\Lambda \longrightarrow \textup{Div}_\Lambda\longrightarrow \textup{Pic}_\infty\longrightarrow 0. \]
\end{corollary}
\begin{proof}
    We have a natural tautological exact sequence
    \[0\longrightarrow \textup{Pr}(X_\infty)\longrightarrow \textup{Div}^0(X_\infty)\longrightarrow J(X_\infty)\longrightarrow 0. \]
    If we tensor this sequence with $\Lambda$ we get a right exact sequence of $\Lambda$-modules
    \[\textup{Pr}_\Lambda \longrightarrow \textup{Div}^0_\Lambda\longrightarrow J_\infty\longrightarrow 0.\] 
    It remains to see that the left map is in fact injective. By Lemma \ref{lem:free} we know that $\textup{Pr}_\Lambda$ embeds into $\textup{Div}_\Lambda$. As the generators $p_{i,0}$ of $\textup{Pr}_\Lambda$ all lie in $M_\Lambda$ we see that $\textup{Pr}_\Lambda$ actually embeds into ${M_\Lambda \subseteq \textup{Div}_\Lambda}$ which is in turn isomorphic to $\textup{Div}^0_\Lambda$ by Lemma \ref{lemma:iso}.
    
    The second sequence can be proven similarly.
\end{proof}
\begin{lemma} We have isomorphisms
\label{tensor-proj-lim}
    \[J_\infty=\varprojlim_n J(X_n), \quad \textup{Pic}_\infty=\varprojlim_n\textup{Pic}(X_n). \]
\end{lemma}
\begin{proof}
    We have an exact sequence
    \[0\longrightarrow \varprojlim_n \textup{Pr}(X_n)\longrightarrow \varprojlim_n \textup{Div}^0(X_n)\longrightarrow \varprojlim_n J(X_n).\]
    Note that 
    \begin{itemize}
        \item $\varprojlim_n J(X_n)\cong \varprojlim_n M_\Lambda/(\textup{Pr}_\Lambda+I_n\textup{Div}_\Lambda)\cong M_\Lambda/\textup{Pr}_\Lambda$ by Lemma \ref{lemma:J}.
        \item $\varprojlim_n \textup{Pr}(X_n)\cong \textup{Pr}_\Lambda$ as each $\textup{Pr}(X_n)$ (and $\textup{Pr}_\Lambda$) is generated by the $p_{i,0}$ as a $\Z_p[G^{(n)}]$-module (as a $\Lambda$-module).
        \item $\varprojlim_n \textup{Div}^0(X_n)\subseteq \textup{Div}_\Lambda$ coincides with $M_\Lambda \cong\textup{Div}^0_\Lambda$ in view of Lemma~\ref{lemma:iso}.  
    \end{itemize}
    Using these three observations we see that the above sequence is indeed right exact. Using the previous corollary gives the first claim. The proof for $\textup{Pic}_\infty$ works analogously. 
\end{proof}
\begin{corollary} \label{cor:inclusion} 
    We have a natural injection $J_\infty\hookrightarrow \textup{Pic}_\infty$.
\end{corollary}
\begin{proof}
    For each finite level $n$ we have a natural injection $J(X_n)\hookrightarrow \textup{Pic}(X_n)$. Thus, we have an injection
    \[\varprojlim_nJ(X_n)\hookrightarrow \varprojlim_n\textup{Pic}(X_n).\]
    The claim now follows from Lemma \ref{tensor-proj-lim}.
\end{proof}

\begin{thm} \label{thm:no-pseudo-null} 
   $\textup{Pic}_\infty$ and $J_\infty$ are $\Lambda$-torsion modules which do not contain any non-trivial pseudo-null $\Lambda$-submodules. 
\end{thm} 
\begin{proof} 
   In view of Lemma~\ref{lem:free} we have an exact sequence of the form 
   \[ 0 \longrightarrow \Lambda^a \longrightarrow \Lambda^a \longrightarrow \textup{Pic}_\infty \longrightarrow 0 \] 
   with ${a = |X|= |X_0|}$. Therefore $\textup{Pic}_\infty$ is $\Lambda$-torsion in view of the additivity of ranks on short exact sequences. 
   
   Note that every free $\Lambda$-module is projective. Therefore the projective dimension of $\textup{Pic}_\infty$ is at most one, and \cite[Theorem~2.8 and Proposition~2.7]{ochi-venjakob2} imply that this module does not contain any non-trivial pseudo-null submodules. The assertion for $J_\infty$ follows in view of Corollary~\ref{cor:inclusion}. 
\end{proof}

\section{Ihara L-functions and reduced norms} \label{section:Ihara-L-functions} 
In this section we will introduce the Ihara zeta-function of finite graphs $X$. If $Y/X$ is a Galois covering one can express the Ihara L-zeta function of $Y$ as a product of the Ihara zeta-function of $X$ and certain L-function. If $Y$ is the derived graph of a voltage assignment $\alpha$ these L-functions can be expressed in terms of irreducible representations of $\Gal(Y/X)$ and the voltage assignment $\alpha$. We will summarize these relations below.

For the remainder of this section let $H$ be a finite group and let $Y/X$ be a Galois cover of finite connected graphs with group $H$. Let $A$ be the adjacency matrix of $X$ and let $D$ be the degree matrix. This means that $A = (a_{ij})$ with 
\[ a_{ij} = \begin{cases} 
   \textup{twice the number of undirected loops at $v_i$} & i = j, \\ 
   |E(v_i, v_j)| & i \ne j, 
\end{cases} \] 
and that $D$ is the diagonal matrix with entries $\deg(v_i)$, ${1 \le i \le m}$, where we recall that ${V(X) = \{v_1, \ldots, v_m\}}$.

The Ihara zeta function attached to $X$ is defined by 
\[\zeta_X(u)^{-1}=(1-u^2)^{-\chi(X)}\det(I-Au+(D-I)u^2),\]
where ${\chi(X) = |V(X)| - |E(X)|}$ is the Euler characteristic of $X$. 
Let $\rho$ be an irreducible representation of $H$ of dimension $d$. Let $L/\Q_p$ be a finite extension which is generated by the values of the character of $\rho$. For every $\sigma\in H$ define the matrix $A(\sigma)$ given by 
\[a_{i,j}(\sigma)=\begin{cases}
    2\vert \{\textup{loops at } (v_i,1)\}\vert & \textup{if } i=j,\sigma=1, \\
    \vert \{ \textup{edges from $(v_i,1)$ to $(v_j,\sigma)$}\}\vert & \textup{otherwise}.
\end{cases}\]
Define $A_\rho=\sum_{\sigma\in H}A(\sigma)\otimes \rho(\sigma)\in Mat_{md \times md}(L)$ and let $D_\rho=D\otimes \textup{Id}_d$. Then we define
\[L(\rho,Y/X,u)^{-1}=(1-u^2)^{-d\chi(X)}\det(I-A_\rho u+(D_\rho-I)u^2).\] 
Let \[h(\rho,X/Y,u)=\det(I-A_\rho u-(I-D_\rho)u^2).\] 

Now suppose that $Y/X$ is the derived graph of a voltage assignment $\alpha$. Define $A_\alpha$ as the following matrix
\[\alpha_{i,j}=\begin{cases}
    \sum _{\{ e_k\in E(v_i,v_j)\}}\alpha(e_k) & \textup{if } i\neq j, \\
    
    \sum_{\{ e_k\in E(v_i,v_i)\}}\alpha(e_k)+\alpha(e_k)^{-1} & \textup{if } i=j. 
\end{cases}\] 
Let $K$ be a finite extension of $\Q_p$. Then $\rho$ induces an algebra homomorphism
\[\rho \colon \textup{Mat}_{m\times m}(K[H])\longrightarrow \textup{Mat}_{md\times md}(K)\]
which is given by evaluating the matrix entries $m_{i,j}$ at $\rho$. 

\begin{lemma} We have 
    \[\rho(D-A_\alpha)=D_\rho-A_\rho.\]
\end{lemma}
\begin{proof}  We can understand both matrices $D_\rho$ and $A_\rho$ as $m\times m$ matrices whose entries are $d\times d$ matrices. 
Note that for $x\in K$ we have $\rho(x)=x\textup{Id}_d$. Thus, $\rho(D)=D_\rho$. It remains to show that $\rho(A_\alpha)=A_\rho$. Assume first that $i\neq j$. Then 
\begin{align*}\rho(\alpha_{i,j})&=\sum _{\{ e_k\in E(v_i,v_j)\}}\rho(\alpha(e_k))\\
&=\sum_{\sigma \in H}\vert\{\textup{ edges from $(v_i,1)$ to $(v_j,\sigma)$}\}\vert\rho(\sigma)\\
&=\sum_{\sigma\in H} a_{i,j}(\sigma)\rho(\sigma)\end{align*}
If $i=j$ we have 
\begin{align*}
    \rho(\alpha_{i,i})&=\rho(\sum_{\{e_k \in E(v_i,v_i)\}} \alpha(e_k)+ \alpha(e_k)^{-1}) \\ 
    &=\sum_{\sigma \in H,\sigma \neq 1}\vert \{\textup{edges from $(v_i,1)$ to $(v_i,\sigma)$}\}\vert \rho(\sigma)+\rho(1) 2\vert \{\textup{loops at $(v_i,1)$}\}\vert\\
    &=\sum_{\sigma \in H}a_{i,j}(\sigma)\rho(\sigma),
\end{align*} which concludes the claim.
\end{proof}
As an immediate corollary we obtain
\begin{corollary}
\label{cor:interpolation}
    \[h(\rho,X/Y,1)=\det(\rho(D-A_\alpha^t)).\]
\end{corollary}
\begin{rem}
We consider $D-A_\alpha^t$ here as this is the matrix that will occur later in the description of the Picard groups of our graphs. The above statement is also correct without the transposed.
\end{rem}
So if we interpret $h(\rho,X/Y,u)$ as the algebraic part of the Ihara $L$-function, then in some sense $D-A_\alpha^t$ interpolates the values at $u=1$. We will make this notion a little more precise in the following.

\begin{def1} Let $L$ be any field of characteristic zero, and recall that $H$ is a finite group. Then $L[H]=\bigoplus_{i=1}^t \textup{Mat}_{s_i\times s_i}(D_i)$ for skew fields $D_i$. Let $F_i$ be the center of $D_i$ and write $Z(L[H])$ for the center of $L[H]$. We define the \emph{reduced norm} 
\[\textup{Nrd}\colon L[H]\longrightarrow Z(L[H])=\bigoplus_{i=1}^tF_i\] as in \cite[§~7D]{curtis-reiner1}. If each of the $D_i$ is actually a field, then the reduced norm coincides with the determinant componentwise. Note: in our applications, $H$ will be a finite $p$-group and therefore 
\[ Z(L[H]) \cong \bigoplus_{\chi / \sim} L(\chi) \]
for any local field $L$ (see \cite{roquette}), where $L(\chi)$ denotes the extension generated by the values $\chi(h)$, ${h \in H}$. Here the sum runs over all characters of irreducible representations of $H$, where the representations are counted modulo isomorphy. 

The reduced norm does not map ${\Z_p[H] \subseteq \Q_p[H]}$ to the center of $\Z_p[H]$, but to the center of a maximal $\Z_p$-order ${\Lambda' \subseteq \Q_p[H]}$. $Z(\Lambda')$ is isomorphic to $\bigoplus_{\chi/\sim} \Z_p(\chi)$. 
\end{def1} 

Let $K/\Q_p$ be a splitting field of $\Q_p[H]$, i.e. a finite extension such that 
\begin{align} \label{eq:wedderburn} K[H]\cong \bigoplus_{i=1}^s \textup{Mat}_{d_i\times d_i}(K), \end{align} 
for suitable integers $d_i$ (by the above, for a finite $p$-group $H$ it will suffice to consider an extension field $K$ of $\Q_p$ which contains the values of all the characters $\chi$). 

If $\rho$ is an irreducible representation of $H$, then $K[H]/\ker(\rho)$ is a simple $K[H]$-algebra, i.e. there exists an index $i$ such that ${K[H]/\ker(\rho)=\textup{Mat}_{d_i\times d_i}(K)}$ and $\rho$ is isomorphic to the representation given by taking the quotient $$K[H]\longrightarrow K[H]/\ker(\rho).$$ 
As the projections to the different components are pairwise non-isomorphic the index $i$ is unique. Note that in this case $\dim(\rho)=d_i$. We will say that $i$ is the index corresponding to $\rho$.  
Note furthermore that $K$ embeds naturally into $\textup{Mat}_{d_i \times d_i}(K)$ and we will always understand $\textup{Nrd}$ to take values in $K[H]$ via this embedding and the isomorphism~\eqref{eq:wedderburn}. If ${x\in \Z_p[H]}$ then the value $\textup{Nrd}(x)$ does not depend on the choice of $K$ but it does not necessarily land in ${\bigoplus_i \Z_p}$. For most of our computations it will be easier to work over a field $K$ such that we have a decomposition as in \eqref{eq:wedderburn}. Later on we will still interpret the reduced norm as taking values in the center of $\Q_p[H]$. The following lemma explains the relation between the determinant and the reduced  norm.
\begin{lemma}
\label{lemma-interpolation}
    \[\det(\rho(D-A_\alpha^t))I_d=\rho(\textup{Nrd}(D-A_\alpha^t)).\]
\end{lemma} 
\begin{proof}
    Write $D-A_\alpha^t\in \textup{Mat}_{m\times m}(K[H])$ as $\sum M_j$, where $M_j\in \textup{Mat}_{d_jm\times d_jm}(K)$. Let $i$ be the index corresponding to $\rho$. Then 
    \[\det(\rho(D-A_\alpha^t))I_{d_i}=\det(M_i)I_{d_i}=\rho(\textup{Nrd}(\sum M_j))=\rho(\textup{Nrd}(D-A_\alpha^t)).\]
\end{proof} 

\begin{lemma} \label{lemma:blockdets} 
  Let ${B \in \textup{Mat}_{b \times b}(\Z_p[H])}$ be a matrix of the form 
  \[ B = \begin{pmatrix} C & 0 \\ \star & a \end{pmatrix}, \] 
  where ${a \in \Z_p[H]}$ and $C$ is a ${(b-1) \times (b-1)}$-matrix. 

  Then ${\textup{Nrd}(B) = \textup{Nrd}(C) \cdot \textup{Nrd}(a)}$. 
\end{lemma} 
\begin{proof} 
  It suffices to prove that ${\rho(\Nrd(B)) = \rho(\Nrd(C)) \cdot \rho(\Nrd(a))}$ for each irreducible representation $\rho$ of $H$. As in the previous lemma we can show that 
  \[ \rho \circ \Nrd = \det \circ \rho.\] 
  Let ${\dim(\rho) = d}$. Then 
  \[ \rho(B) = \begin{pmatrix} \rho(C) & 0 \\ \star & \rho(a) \end{pmatrix},\] 
  where $\rho(C)$ is a ${d(b-1) \times d(b-1)}$-matrix and $\rho(a)$ is a ${d \times d}$-matrix with entries in some field $F$. Clearly, 
  \[ \det(\rho(B)) = \det(\rho(C)) \cdot \det(\rho(a)), \] 
  which implies the claim. 
\end{proof} 

\begin{lemma}
\label{compatibility-of-reduced-norms}
    Let $H'$ be a quotient of $H$. For $P\in \{H,H'\}$ let $\textup{Nrd}_P$ denote the reduced norm on $K[P]$. Let $\pi\colon K[H]\longrightarrow K[H']$ be the natural projection. Then
    \[\pi \circ \textup{Nrd}_H=\textup{Nrd}_{H'}\circ \pi.\]
\end{lemma}
\begin{proof}
    Note that $\pi$ is a homomorphism of $K$-algebras. As $K[H]$ is semisimple the kernel of $\pi$ can be written as a direct sum of simple $K[H]$-subalgebras. The simple subalgebras are precisely the $\textup{Mat}_{d_i\times d_i}(K)$. Thus there is a finite subset $J\subseteq \{1,\dots ,s\}$ such that $K[H']\cong \bigoplus_{i\in J}\textup{Mat}_{d_i\times d_i}(K)$. Note that $J$ contains exactly the indices corresponding to representations that factor through $H'$. From this the claim is immediate. 
\end{proof}

\section{Basic properties of Fitting orders and Fitting invariants} \label{section:Fitting} 
In this section we introduce the definition and basic properties of Fitting orders and Fitting ideals. We mainly follow \cite[Section~2.1]{nickel} we give the following definitions. 
\begin{def1}
Let $\mathcal{O}$ be an integrally closed, noetherian, commutative, complete local domain with field of quotients $F$. Let $A$ be a finite dimensional separable $F$-algebra and let $\Lambda$ be an $\mathcal{O}$-order in $A$. Then we call $\Lambda$ a \emph{Fitting order} over $\Ok$.
\end{def1}
\begin{example}
Let $G$ be a finite group and let $\mathcal{O}$ be the ring of integers in a finite extension $K$ of $\Q_p$. Then $\mathcal{O}[G]$ is a Fitting order. 
\end{example}

Let $G$ be a uniform $p$-adic Lie-group. Then $\Z_p[G^{(n)}]$ is a Fitting order over $\Z_p$ for each ${n \in \N}$. For every $n$ we can choose a finite extension $K^{(n)}$ of $\Q_p$ such that $\Q_p[G^{(n)}]=\bigoplus_{i=1}^s\textup{Mat}_{d_i\times d_i}(K^{(n)})$ and ${K^{(n)}\subseteq K^{(n+1)}}$. Let ${K^{(\infty)}=\bigcup K^{(n)}}$. 

\begin{def1} 
 Let $\pi_n\colon \Z_p\llbracket G\rrbracket \longrightarrow \Z_p[G^{(n)}]$ be the natural projection.
    Let ${M\in \textup{Mat}_{d\times d}(\Z_p\llbracket G\rrbracket )}$. We define $\textup{Nrd}(M)\in Z(\Q_p\llbracket G\rrbracket )$ as the projective limit of the elements $\textup{Nrd}_{G^{(n)}}(\pi_n(M))$ (it follows from Lemma~\ref{compatibility-of-reduced-norms} that these elements indeed yield an element in ${\varprojlim_n Z(\Q_p[G^{(n)}])}$). 
\end{def1}

This definition might seem artificial and unnecessarily complicated. The reason that we define the reduced norm as a projective limit is that $\mathcal{O}\llbracket G\rrbracket $ is in general not a Fitting order:   
\begin{rem}
Let $G$ be a $p$-adic Lie group and let $Z$ be its center. Assume that $G/Z$ is infinite. The maximal subalgebra lying in the center of $\Z_p\llbracket G\rrbracket $ is $\Z_p\llbracket Z\rrbracket $. So, if we want to see $\Z_p\llbracket G\rrbracket $ as a Fitting order, the corresponding ring $\mathfrak{O}$ has to lie inside $\Z_p\llbracket Z\rrbracket $. Let $F$ be the quotient field of $\Z_p\llbracket Z\rrbracket $. If $\Z_p\llbracket G\rrbracket $ is a Fitting order then $\Z_p\llbracket G\rrbracket $ is contained in some finite-dimensional $F$-algebra. Let $t$ be a non-negative integer and let $h_1,\dots, h_t$ be elements in $G$ with distinct images  in $G/Z$. Such elements exist as $G/Z$ is infinite. Then $h_1,\dots, h_t$ are independent over $F$. As $t$ was arbitrary we see that $\Z_p\llbracket G\rrbracket $ cannot be contained in a finite-dimensional $F$-algebra.
\end{rem}

The definition using projective limits has furthermore the advantage that it is compatible with "interpolation" properties. 

\begin{lemma}
Let $X_\infty$ be the derived graph of a voltage assignment ${\alpha \colon S \longrightarrow G}$. Fix ${n \in \N}$, and let $D$ and $A_\alpha$ be attached to the Galois subcover $X_n/X$, as in Section~\ref{section:Ihara-L-functions}. Then $\textup{Nrd}(D-A_\alpha^t)$ interpolates $h(\rho,X_n/X,1)$ in the following sense. Let $\rho$ be a representation of dimension $d$ that factors through $G^{(n)}$. Then 
\[h(\rho,X_n/X,1)I_{d}=\rho(\textup{Nrd}(D-A_\alpha^t)). \]
\end{lemma}
\begin{proof}
  The representation $\rho$ induces a homomorphism $\Z_p\llbracket G\rrbracket \longrightarrow \textup{Mat}_{d\times d}(K^{(\infty)})$ that factors through $\Z_p[G^{(n)}]$. Therefore it suffices to determine 
  \[ \rho(\textup{Nrd}(\pi_n(D-A_\alpha^t))).\] 
  The claim now follows from Lemma \ref{lemma-interpolation} and Corollary~\ref{cor:interpolation}. 
\end{proof}

\begin{def1} Recall that we write $Z(R)$ for the center of any ring $R$. For each $n$ we define a $Z(\Z_p[G^{(n)}])$-submodule of ${\Q_p[G^{(n)}]}$ by 
\[ I(\Z_p[G^{(n)}]) = \langle \textup{Nrd}(H) \mid H\in \textup{Mat}_{b\times b}(\Z_p[G^{(n)}]), b \in \N \rangle_{Z(\Z_p[G^{(n)}])}. \] 
\end{def1} 
Note that $I(\Z_p[G^{(n)}])$ is actually a \emph{subring} of $\Q_p[G^{(n)}]$, ${n \in \N}$, which contains the center of $\Z_p[G^{(n)}]$. If $\Lambda'_n$ is a maximal $\Z_p$-order in $\Q_p[G^{(n)}]$ containing $\Z_p[G^{(n)}]$, then ${I(\Z_p[G^{(n)}])\subseteq Z(\Lambda_n')}=I(\Lambda_n')$ and the index of this inclusion is finite (as $\Lambda'_n$ and $\Z_p[G^{(n)}]$ have the same $\Z_p$-rank).  Furthermore, $I(\Lambda_n')$ is a direct sum of abelian local rings (see \cite[Section~2.2]{nickel}).
\begin{lemma} \label{lemma:fuerdich} 
$I(\Z_p[G^{(n)}])$ is a compact semi-local ring (i.e. it has a finite number of maximal left ideals). 
\end{lemma}
\begin{proof} 
By definition $I(\Z_p[G^{(n)}])$ is a finitely generated $\Z_p$-module. This proves compactness. In view of \cite[Proposition~2.6]{Lam}, the $\Z_p$-algebra $I(\Z_p[G^{(n)}])$ is semi-local. 
\end{proof}

\begin{def1} \label{def:fitt_endl} 
Let $N$ be a finitely generated $\Z_p[G^{(n)}]$ module. For any presentation 
\[ h \colon \Z_p[G^{(n)}]^{a}\longrightarrow \Z_p[G^{(n)}]^{b}\longrightarrow N\]
of $N$, we write $a(h)$ and $b(h)$ to make the dependence clear.
If ${b(h)>a(h)}$ then we define $\textup{Fitt}_{\Z_p[G^{(n)}]}(N)=\{0\}$. Otherwise, we let \[\textup{Fitt}_{\Z_p[G^{(n)}]}(N)=\langle \textup{Nrd}(H)\mid H \textup{ is a $b(h)\times b(h)$ submatrix of $h$ for some $h$}\rangle_{I(\Z_p[G^{(n)}])}. \] 
We call this ideal the \emph{Fitting invariant} of $N$ over $\Z_p[G^{(n)}]$. 
\end{def1} 
\begin{rem} \label{rem:dependance} It turns out that the ideal 
\[\textup{Fitt}^{(h)}_{\Z_p[G^{(n)}]}(N):=\langle \textup{Nrd}(H)\mid H \textup{ is a $b(h)\times b(h)$ submatrix of $h$}\rangle_{I(\Z_p[G^{(n)}])} \] 
generated by the reduced norms of the submatrices for some fixed presentation $h$ of $N$ 
in general depends on the chosen presentation $h$ (see \cite[Example~2.11]{nickel}). We circumvent this dependency by considering all presentations $h$ of $N$ at once (this is what is called the \emph{maximal Fitting invariant} in \cite[Theorem~2.13]{nickel}). 

On the other hand, if ${a(h) = b(h)}$, i.e. if $N$ is \emph{quadratically presented} (which will be the case in our applications), then the ideal $\textup{Fitt}^{(h)}_{\Z_p[G^{(n)}]}(N)$ is independent from $h$ in the following sense: if $h$ is any finite presentation of $N$ with $a(h)=b(H)$, then 
\[ \textup{Fitt}^{(h)}_{\Z_p[G^{(n)}]}(N) = \textup{Fitt}_{\Z_p[G^{(n)}]}(N) \] 
(see \cite[Proposition~2.17]{nickel}). 
\end{rem} 
Now we want to define Fitting invariants of $\Lambda$-modules. Before we can state the actual definition we need a few auxiliary lemmas.
\begin{lemma} For $m\le n$ we denote, by abuse of notation, the natural projection ${\Z_p[G^{(n)}] \longrightarrow \Z_p[G^{(m)}]}$ by $\pi_m$. Then 
    \[\pi_m(I(\Z_p[G^{(n)}]))\subseteq I(\Z_p[G^{(m)}]). \]
\end{lemma}
\begin{proof}
    This is proven in \cite[Lemma~1.3(i)]{ellerbrock-nickel}. For the convenience of the reader, we recapitulate the short proof. Let $H\in \textup{Mat}_{b\times b}(\Z_p[G^{(n)}])$. By Lemma \ref{compatibility-of-reduced-norms} we see that
    \[\pi_m(\textup{Nrd}(H))=\textup{Nrd}(\pi_m(H)) \in I(\Z_p[G^{(m)}]).\]
     As $\pi_m(Z(\Z_p[G^{(n)}]))\subseteq Z(\Z_p[G^{(m)}])$, the claim follows. 
\end{proof}
\begin{lemma}\label{representations-are-compatible} Let $h$ be a presentation of a finitely generated $\Lambda$-module $N$. Then $h$ induces a presentation of ${\pi_n(N) := N/I_n N}$ as a $\Z_p[G^{(n)}]$-module.
\end{lemma}
\begin{proof}
Recall that ${I_n = \ker(\pi_n \colon \Lambda \longrightarrow \Z_p[G^{(n)}])}$. Let the presentation $h$ be given by 
\[\Lambda^a\longrightarrow \Lambda^b\longrightarrow N.\]
By abuse of notation we write $h$ for the corresponding map $h\colon \Lambda^b\longrightarrow N$.
We obtain a surjection ${\pi_n(h) \colon (\Lambda/\ker(\pi_n))^b\longrightarrow \pi_n(N)}$. It remains to determine the kernel of $\pi_n(h)$. The pre-image of $I_nN$ under $h$ is clearly $(I_n\Lambda)^b+\ker(h)=(I_n\Lambda)^b+\textup{Im}(\Lambda^a)$. It follows that the kernel of $\pi_n(h)$ is the image of $(\Lambda/I_n)^a$ in $(\Lambda/I_n)^b$. 
\end{proof}
 \begin{def1}
 Let $I(\Lambda)$ be the projective limit $\varprojlim_{n}I(\Z_p[G^{(n)}])$. 
 \end{def1} 

 Now we can define the Fitting invariants of $\Lambda$-modules. 
 \begin{def1} For every presentation 
 \[h\colon \Lambda^{a}\longrightarrow \Lambda^{b}\longrightarrow N,\]
 we write ${a = a(h)}$ and ${b = b(h)}$ to make the dependency clear. Now we define 
 \[\textup{Fitt}_\Lambda(N)=\langle \textup{Nrd}(H)\mid H \textup{ is a $b(h)\times b(h)$ submatrix of $h$ for some $h$}\rangle_{I(\Lambda)}.\]
 \end{def1}
Note that it follows from Definition~\ref{def:fitt_endl} and (the proof of) Lemma~\ref{quadratical represented} below that $\textup{Fitt}_\Lambda(N) = \{0\}$ if ${b(h) > a(h)}$. 
 
 \begin{lemma}
 \label{compact}
     Let $N$ be a finitely generated $\Lambda$-module. Then 
     $$\varprojlim_n \textup{Fitt}_{\Z_p[G^{(n)}]}(\pi_n(N))$$ 
     is a compact $I(\Lambda)$-module.
 \end{lemma}
 \begin{proof}
     It suffices to show that each $\textup{Fitt}_{\Z_p[G^{(n)}]}(\pi_n(N))$ is a compact $I(\Z_p[G^{(n)}])$-module. Each single $I(\Z_p[G^{(n)}])$ is a compact semi-local ring by Lemma~\ref{lemma:fuerdich} and the Fitting invariant $\textup{Fitt}_{\Z_p[G^{(n)}]}(\pi_n(N))$ is finitely generated as a $I(\Z_p[G^{(n)}])$-module. In particular, each single $\textup{Fitt}_{\Z_p[G^{(n)}]}(\pi_n(N))$ is compact.
 \end{proof}
 \begin{lemma}
 \label{quadratical represented}
    We always have the identity
     \[\textup{Fitt}_\Lambda(N)=\varprojlim_n \textup{Fitt}_{\Z_p[G^{(n)}]}(\pi_n(N)).\] 
 \end{lemma}
 \begin{proof} 
 We will first show  the inclusion $\subseteq$. Consider a presentation
 \[\Lambda^a\longrightarrow \Lambda^b\longrightarrow N,\]
 which is represented by the matrix $H$. Let ${n \in \N}$ be arbitrary. Then $\pi_n(H)$ describes a presentation of $\pi_n(N)$ as a $\Z_p[G^{(n)}]$-module by Lemma \ref{representations-are-compatible}. Using Lemma~\ref{compatibility-of-reduced-norms} it follows that
 \[\textup{Fitt}_\Lambda(N)\subseteq \varprojlim_{n}\textup{Fitt}_{\Z_p[G^{(n)}]}(\pi_n(N)).\]
 
 Assume now that we have a presentation 
 \[ \Z_p[G^{(n)}]^a\longrightarrow \Z_p[G^{(n)}]^b \stackrel{h'}{\longrightarrow} \pi_n(N).\]
 This presentation induces  a surjection $\psi_n\colon\Lambda^b\longrightarrow \pi_n(N)$. As $N\longrightarrow \pi_n(N)$ is surjective and $\Lambda^b$ is projective we obtain a map $\psi\colon \Lambda^b\longrightarrow N$ such that ${\pi_n\circ \psi=\psi_n}$. Recall that $\pi_n(N)=N/I_nN$. 
 By Lemma \ref{nakayama} we see that $\psi$ is surjective. Furthermore, $(I_n\Lambda)^b$ surjects to the kernel of $\pi_n\colon N\longrightarrow \pi_n(N)$ by definition. Then the snake lemma implies that $\ker(\psi)$ surjects onto $\ker(h')$. Therefore the homomorphism $\psi$ induces a presentation $\psi_m$ of $\pi_m(N)$ for every ${m \ge n}$. 
 
This argument works for every $n$ and we see that the image of $\textup{Fitt}_\Lambda(N)$ lies dense in $\varprojlim_n\textup{Fitt}_{\Z_p[G^{(n)}]}(\pi_n(N))$. Since it is also closed, and as the projective limit $\varprojlim_n \textup{Fitt}_{\Z_p[G^{(n)}]}(\pi_n(N))$  is compact  by Lemma \ref{compact}, we obtain the claim. 
 \end{proof}

Finally, we want to compare Fitting invariants and annihilator ideals. For a finitely generated module $N$ over a \emph{commutative} ring $R$, it is well-known that the Fitting ideal of $N$ over $R$ is contained in the annihilator ideal, and that a suitable power of the annihilator ideal is contained in the Fitting ideal (it suffices to raise the annihilator ideal to the power $g$ if $N$ can be generated over $R$ by $g$ elements), see for example \cite[Chapter~3, Theorem~5]{northcott}. Therefore the Fitting ideal carries much information about the annihilators of $N$. 

In the non-commutative setting everything is a bit more subtle. First of all, even if $N$ is a torsion $\Lambda$-module, where $\Lambda$ now shall have the usual meaning (i.e. ${\Lambda = \Z_p\llbracket G\rrbracket }$ for some uniform $p$-group $G$), then it is not clear whether there exists a non-trivial global annihilator of $N$, i.e. the annihilator ideal 
\[ \textup{Ann}_\Lambda(N) = \{ x \in \Lambda \mid x \cdot n = 0 \, \forall n \in N \}\] 
could be just the zero ideal. 

The main goal of the remainder of this section is the proof of the following 
\begin{prop} \label{prop:Fittingvsann} 
  Let $N$ be a finitely generated $\Lambda$-module which has a quadratic presentation ${\Lambda^l \stackrel{h}{\longrightarrow} \Lambda^l \twoheadrightarrow N}$. 

  Then $\Nrd(\textup{Ann}(N))^l \subseteq \textup{Fitt}_\Lambda(N)$. 
\end{prop} 
\begin{proof} 
  We mimic the proof of \cite[Proposition~20.7b.]{eisenbud}. In order to prove the proposition, we introduce the notion of \emph{higher Fitting invariants}. 
  \begin{def1} For every presentation 
 \[h\colon \Lambda^{a}\longrightarrow \Lambda^{b}\longrightarrow N,\]
 we write ${a = a(h)}$ and ${b = b(h)}$ to make the dependency clear. For ${j \in \N}$ we define the \emph{$j$-th Fitting invariant} of $N$ with respect to $h$ as 
 \[\textup{Fitt}_\Lambda^j(N, h)=\langle \textup{Nrd}(H)\mid H \textup{ is a $(b(h) - j)\times (b(h)-j)$ submatrix of $h$}\rangle_{I(\Lambda)}.\] 
 Here we define $\textup{Nrd}(H) = 1$ if ${j \ge b(h)}$. 
 \end{def1}
 \begin{rem} 
   For $j = 0$ we just obtain the classical Fitting invariant of $N$ with respect to $h$. If $h$ is a quadratic presentation of $N$, then ${\textup{Fitt}_\Lambda^0(N,h) = \textup{Fitt}_\Lambda(N)}$. 
 \end{rem} 

 \begin{rem} It follows from the definitions that for a quadratically presented $\Lambda$-module $N$ with presentation ${h \colon \Lambda^b \longrightarrow \Lambda^b \longrightarrow N}$, we have 
  \[ \textup{Fitt}_\Lambda^j(N, h) = I(\Lambda)\]
  if ${j \ge b(h)}$. 
  \end{rem} 
  
  Now we return to the proof of Proposition~\ref{prop:Fittingvsann}. Fix ${a \in \textup{Ann}(N)}$.   Let $h$ be as in the statement of the proposition, and consider the presentations 
  \[ h + a \cdot \id_{rl} \colon \Lambda^l \oplus \Lambda^{rl} \longrightarrow \Lambda^l \longrightarrow N \] 
  of $N$, ${r \in \N}$. The corresponding matrix is the ${l \times (r+1)l}$-matrix ${(h \mid a I_{rl})}$, where $a I_{rl}$ shall mean the ${rl \times rl}$-diagonal matrix which is $a$ times the identity matrix.
  
  We will first prove that 
  \[ \textup{Nrd}(a) \cdot \textup{Fitt}_\Lambda^j(N,h + a \cdot \id_{rl}) \subseteq \textup{Fitt}_\Lambda^{j-1}(N,h + a \cdot \id_{(r+1)l})\] 
  for every ${j,r \in \N}$.

  To this purpose, let $H$ be a ${s \times s}$-submatrix of ${(h \mid a I_{rl})}$, ${s \in \N}$. Choose a row which is not contained in $H$ and let ${H \oplus a}$ be the ${(s+1)\times (s+1)}$-submatrix of ${(h \mid a I_{(r+1)l})}$ of the form 
  \[ H \oplus a = \begin{pmatrix} H & 0 \\ \star & a\end{pmatrix}. \] 
  It follows from Lemma~\ref{lemma:blockdets} that 
  \[ \textup{Nrd}(H \oplus a) = \textup{Nrd}(H) \cdot \textup{Nrd}(a). \] 
  Now the left hand side is contained in the Fitting ideal $\textup{Fitt}^{l-(s+1)}_\Lambda(N, h \oplus a \cdot \id_{(r+1)l})$. Since $H$ was an arbitrary submatrix of the presentation ${h + a \cdot \id_{rl}}$ of $N$, it follows that indeed 
  \[ \textup{Nrd}(a) \cdot \textup{Fitt}_\Lambda^j(N,h + a \cdot \id_{rl}) \subseteq \textup{Fitt}_\Lambda^{j-1}(N,h+ a \cdot \id_{(r+1)l}), \] 
  where ${j =l - s}$. 
Now we conclude the proof of Proposition~\ref{prop:Fittingvsann}. First, recall that ${\textup{Fitt}_\Lambda^l(N,h) = I(\Lambda)}$. Therefore 
\[ \textup{Nrd}(a) \in \textup{Nrd}(a) \cdot I(\Lambda) = \textup{Nrd}(a) \cdot \textup{Fitt}_\Lambda^l(N,h) \subseteq \textup{Fitt}_\Lambda^{l-1}(N,h + a \cdot \id_l). \] 
Then 
\[ \textup{Nrd}(a)^2 \in \textup{Nrd}(a) \cdot \textup{Fitt}_\Lambda^{l-1}(N,h+ a \cdot \id_l) \subseteq \textup{Fitt}_\Lambda^{l-2}(N,h + a \cdot \id_{2l}).\] 
Inductively, we obtain that 
\[ \textup{Nrd}(a)^l \in \textup{Fitt}_\Lambda^0(N,h + a \cdot \id_{l^2}) \subseteq  \textup{Fitt}_\Lambda(N). \]
Since ${a \in \textup{Ann}(N)}$ was arbitrary, this proves the claim. 
\end{proof} 

\section{A non-commutative main conjecture} \label{section:main_conjecture} 
We finally have all pieces in place to prove the main result of this article, i.e. a non-commutative main conjecture for Jacobians of graphs.
Assume that $G$ is a uniform $p$-adic Lie group as in the preceding section. Let ${X_\infty}$ be the derived graph of a voltage assignment ${\alpha \colon S \longrightarrow G}$. Recall from Corollary~\ref{cor:J_infty} that we have a short exact sequence 
\[ 0 \longrightarrow \textup{Pr}_\Lambda \longrightarrow \Div_\Lambda \longrightarrow \textup{Pic}_\infty \longrightarrow 0. \] 
Moreover, we have seen in Lemma~\ref{tensor-proj-lim} that 
\[ \textup{Pic}_\infty \cong \varprojlim_n \textup{Pic}(X_n). \] 

Since $\Pr_\Lambda$ is generated by the elements $p_i$, it follows from the definitions that the above presentation of $\textup{Pic}_\infty$ is given by the matrix $(D - A_\alpha)^t=D-A_\alpha^t$. We also know that $\textup{Pic}(X_n)$ is quadratically presented by the matrix $\pi_n(D - A_\alpha^t)$. As an immediate consequence we obtain
\begin{lemma}
\label{maini}
    \[\textup{Fitt}_\Lambda(\textup{Pic}_\infty)= \langle \textup{Nrd}(D-A_\alpha^t)\rangle_{I(\Lambda)}.\]
\end{lemma}
This lemma can be regarded as a main conjecture for $\textup{Pic}_\infty$. If $G$ is abelian it can be rewritten as
\[\textup{Char}_\Lambda(\textup{Pic}_\infty)=\det(D-A_\alpha^t).\]
If furthermore, $d\ge 2$, then $\textup{Pic}_\infty$ and $J_\infty$ have the same characteristic ideal and we recover \cite[Theorem 5.2]{Kleine-Mueller4}. In the remainder of this section we will use Lemma~\ref{maini} to derive results on the Fitting invariant of $J_\infty$, generalizing \cite[Theorem 5.2]{Kleine-Mueller4}. 

\begin{rem} 
Note that in view of Proposition~\ref{prop:Fittingvsann} the above Lemma~\ref{maini} yields information on the global annihilators of $\textup{Pic}_\infty$. To be more precise, if $m$ denotes the number of vertices of the base graph $X$, and ${a \in \Lambda}$ is a global annihilator of $\textup{Pic}_\infty$, then $a^m$ will be contained in the ideal $\langle \textup{Nrd}(D-A_\alpha^t)\rangle_{I(\Lambda)}$. 
\end{rem} 

 It is well-known that Fitting ideals over commutative rings are well-behaved under exact sequences, and similar but somewhat weaker properties are valid for Fitting invariants over Fitting orders (see \cite[Lemma~2.16]{nickel}). We would like to prove something analogous in our setting. 
\begin{lemma}
\label{behavior-under-exact-sequence}
    Let $0\longrightarrow M_1\longrightarrow M_2\longrightarrow M_3\longrightarrow 0$ be an exact sequence of $\Lambda$-modules. Then 
    \[\textup{Fitt}_\Lambda(M_1)\textup{Fitt}_\Lambda(M_3) \subseteq \textup{Fitt}_\Lambda(M_2).\]
    and $\textup{Fitt}_\Lambda(M_2)\subseteq \textup{Fitt}_\Lambda(M_3)$.
\end{lemma}
\begin{proof}
    Let $\Lambda^{a_1}\longrightarrow \Lambda^{b_1}\longrightarrow M_1$ and $\Lambda^{a_2}\longrightarrow \Lambda^{b_2}\longrightarrow M_3$ be presentations represented by the matrices $A_1$ and $A_2$. We clearly get a presentation
    \[\Lambda^{a_1+a_2}\longrightarrow \Lambda^{b_1+b_2}\longrightarrow M_2\]
    given by a matrix of the form
    $A=\begin{pmatrix}
     A_1&*\\0&A_2
    \end{pmatrix}$.
    
    Let $H_1$ and $H_2$ be $b_1\times b_1$ and $b_2\times b_2$ minors of $A_1$ and $A_2$ respectively. Then there is a $(b_1+b_2)\times (b_1+b_2)$-minor $H$ of $A$ of the form
    \[\begin{pmatrix}
    H_1&*\\0&H_2
    \end{pmatrix}.
    \]
    It follows that $\pi_n(H)$, ${n \in \N}$, is of the form
    \[\begin{pmatrix}
    \pi_n(H_1)&*\\0&\pi_n(H_2) 
    \end{pmatrix}. \]
    Therefore 
    \[\textup{Nrd}(H_1)\textup{Nrd}(H_2)=\varprojlim_n\textup{Nrd}\pi_n(H_1)\textup{Nrd}\pi_n(H_2)=\varprojlim_n\textup{Nrd}\pi_n(H)=\textup{Nrd}(H),\]
    which implies $\textup{Fitt}_\Lambda(M_1)\textup{Fitt}_\Lambda(M_3)\subseteq\textup{Fitt}_\Lambda(M_2)$.
     The second claim follows similarly from the corresponding property at finite level.
\end{proof}
We also have the natural exact sequence
\[0\longrightarrow J_\infty\longrightarrow \textup{Pic}_\infty\longrightarrow \Lambda/\textup{Aug}\longrightarrow 0,\]
where $\textup{Aug}$ denotes the augmentation ideal of ${\Lambda = \Z_p\llbracket G\rrbracket }$. 
\begin{lemma}
\label{augmentation}
    \[\textup{Fitt}_\Lambda(\Lambda/\textup{Aug})\supset \langle \textup{Nrd}(g-1)\mid g\in G\rangle_{I(\Lambda)}.\]
\end{lemma}
\begin{proof}
    By Lemma~\ref{quadratical represented} it suffices to consider $\textup{Fitt}_{\Z_p[G^{(n)}]}(\Z_p[G^{(n)}]/\textup{Aug}(G^{(n)}))$ for every ${n \in \N}$, where ${\textup{Aug}(G^{(n)}) \subseteq \Z_p[G^{(n)}]}$ denote the augmentation ideals. Considering the presentation
    \[(\Z_p([G^{(n)}])^{\vert G^{(n)}\vert-1} \longrightarrow \Z_p[G^{(n)}]\longrightarrow \Z_p[G^{(n)}]/\textup{Aug}(G^{(n)}),\] where the right hand map is the canonical projection, we see that ${\textup{Nrd}(g-1)}$ is contained in $\textup{Fitt}_{\Z_p[G^{(n)}]}(\Z_p[G^{(n)}]/\textup{Aug}(G^{(n)}))$ for all $g$. 
    \end{proof}
\begin{lemma}
    \[\textup{Fitt}_\Lambda(J_\infty)\langle \textup{Nrd}(g-1)\mid g\in G\rangle_{I(\Lambda)}\subseteq \langle\textup{Nrd}(D-A_\alpha^t)\rangle_{I(\Lambda)}. \]
\end{lemma}
\begin{proof}
   This is a a direct consequence of Lemmas \ref{maini}, \ref{behavior-under-exact-sequence} and \ref{augmentation}. 
\end{proof}

\section{Relation to algebraic $K$-theory} \label{section:K-theory}  
In this section we restrict to uniform $p$-adic Lie groups $G$ that contain a normal subgroup $H$ such that $G/H\cong \Z_p$, and we will work under the assumption that $\textup{Pic}_\infty$ has the $\mathfrak{M}_H(G)$-property (here, as usual, $X_\infty$ is the derived graph of a voltage assignment $\alpha$ which takes values in $G$). In Iwasawa theory one often formulates (non-commutative) Iwasawa main conjectures in terms of algebraic $K$-theory, see e.g. \cite{ritter-weiss2}, \cite{ritter-weiss}, \cite{kakde}, \cite{5people} or \cite{burns}. In the present section we compare this approach with the one presented above. In the next section we will prove criteria which ensure that the Iwasawa modules $\textup{Pic}_\infty$ and $J_\infty$ satisfy the $\mathfrak{M}_H(G)$-property.  

Let as before $\Lambda=\Z_p\llbracket G\rrbracket $ and let $S$ be the set of all elements $f\in \Lambda$ such that $\Lambda/f$ is finitely generated as $\Z_p\llbracket H\rrbracket $-module. Clearly, $S$ is closed under multiplication. Let $S^*=\bigcup_kp^kS$. As $p$ lies in the center of $\Lambda$, it is immediate that $S^*$ is closed under multiplication as well.

Note that a finitely generated $\Lambda$-module $M$ satisfies the $\mathfrak{M}_H(G)$-property if and only if it is $S$-torsion (see \cite[Proposition~2.3]{5people}). Now we introduce some standard notation from algebraic $K$-theory \cite[Chapters II and III]{k-book}.
\begin{def1}
Let $\mathbf{P}(\Lambda)$ be the category of projective $\Lambda$-modules. Let $\textbf{Ch}^b\mathbf{P}(\Lambda)$ be the category of bounded complexes in $\mathbf{P}(\Lambda)$. Let $\textbf{Ch}^b_{S^*}\mathbf{P}(\Lambda)$ be the Waldhausen subcategory of complexes that are $(S^*)^{-1}\Lambda$ exact. We will also use the notation $K_0(\Lambda \on S^*)$ for $K_0(\textbf{Ch}^b_{S^*}\mathbf{P}(\Lambda))$.
\end{def1}

Let $K_1(\Lambda_{S^*})$ be the $K_1$-group of the localization $(S^*)^{-1}\Lambda$. Then classical algebraic $K$-theory gives us a differential operator
\[\partial\colon K_1(\Lambda_{S^*})\longrightarrow K_0(\Lambda \on S^*).\]
This map sends the class of a homomorphism $\alpha \colon \Lambda^m\longrightarrow \Lambda ^m$ that is invertible as an element in $\textup{Mat}_{m\times m}(\Lambda_{S^*})$ to $[\textup{cone}(\alpha)]$, i.e. to the complex concentrated in $0$ and $1$ given by $-\alpha \colon \Lambda ^m \longrightarrow \Lambda^m$.

By abuse of notation we denote by $\mathfrak{M}_H(G)$ the full subcategory of finitely generated $\Lambda$-modules which satisfy the $\mathfrak{M}_H(G)$-property. Let $M$ be a module satisfying the $\mathfrak{M}_H(G)$-property and let 
\[\alpha\colon\Lambda^a\longrightarrow \Lambda^b \longrightarrow M\]
be a presentation with $a\ge b$. We then obtain a $(S^*)^{-1}\Lambda$-exact complex
\[0\longrightarrow \Lambda^{a-b}\longrightarrow \Lambda^a\longrightarrow \Lambda_{S^*}^b\longrightarrow 0.\]
In the case that $M$ is quadratically presented the first term vanishes and 
\[0\longrightarrow \Lambda^a\longrightarrow \Lambda^a\longrightarrow 0\]
lies in the image of $\partial$.

For the remainder of this section we make the following assumption: 
\begin{ass}
$\textup{Pic}_\infty$ satisfies the $\mathfrak{M}_H(G)$-property. 
\end{ass}
Recall that  each finite-dimensional representation $\rho$ that is defined over some commutative ring $O$ of characteristic zero can be seen as a homomorphism
\[\rho \colon \Lambda \longrightarrow \textup{Mat}_{d\times d}(O),\]
where $d=\dim(\rho)$. Then $\rho$ induces a well-defined homomorphism 
\[\rho \colon K_1(\Lambda)\longrightarrow K_1(\textup{Mat}_{d\times d}(O))=O^\times.\]
The isomorphism $K_1(\textup{Mat}_{d\times d}(O))=O^\times$ is given by the determinant.
Let $L$ be the quotient field of $O$. Then $\rho$ induces a well-defined homomorphism \cite[Section 3]{5people}
\[\rho\colon K_1(\Lambda_{S^*})\longrightarrow L\cup\{\infty\}.\]

We formulate the $K$-theoretic main conjecture (in analogy with the classical case \cite{coates-et-all}) as follows (we use the notation from Section~\ref{section:Ihara-L-functions}). 
\begin{thm} \label{thm:K-theory} 
Let $G$ be a uniform $p$-adic Lie group and let $X_\infty$ be the derived graph over $X$ of a voltage assignment ${\alpha \colon S \longrightarrow G}$. 

The matrix $D-A_\alpha^t$ gives an element in $K_1(\Lambda_{S^*})$, denoted by $[D-A_\alpha^t]$. Then $[D-A_\alpha^t]$ corresponds to a $p$-adic L-function in the following sense: If $\rho$ is a finite-dimensional representation of $G$ factoring through $G^{(n)}$ for some $n$ and being defined over $O$, then 
\[\rho([D-A_\alpha^t])=h(\rho,X_n/X,1) \]
as elements in $L\cup\{\infty\}$, where $h(\rho, X_n/X,1)$ is defined as in Section~\ref{section:Ihara-L-functions}. 
Furthermore, $\partial([D-A_\alpha^t])=[\textup{Pic}_\infty]$.
\end{thm}
\begin{proof} 
Recall that $\textup{Pic}_\infty$ is quadratically presented and that the presentation $h$ is given by the matrix $D-A_\alpha^t$. Note that multiplication with $D-A_\alpha^t$ is injective on $\textup{Div}_\Lambda$. Thus, $D-A_\alpha^t$ is invertible as an element in $\textup{GL}_m(\Lambda_{S^*})$ and therefore it has a well-defined image in $K_1(\Lambda_{S^*})$. Furthermore, $\partial([D-A_\alpha^t])$ is indeed that class of $[\textup{Pic}_\infty]$. So to prove the above theorem, we only have to prove that $[D-A_\alpha^t]$ interpolates the required $L$-value.

\begin{lemma}
\[\rho([D-A_\alpha^t])=h(\rho, X_n/X,1).\]
\end{lemma}
\begin{proof}
This is a reformulation of Corollary \ref{cor:interpolation}.
\end{proof}
This also concludes the proof of Theorem~\ref{thm:K-theory}. \end{proof} 
\begin{rem}
Note that the $K$-theoretic main conjecture provides slightly more information than the one in terms of Fitting invariants: As the reduced norm is trivial on the commutator subgroup, the reduced norm can be defined on $K_1(\Lambda_{S^*})$. So taking reduced norms somehow loses information. On the other hand the Fitting invariant contains less information than the isomorphism class $[\textup{Pic}_\infty]$ in $K_0(\Lambda \textup{ on } S^*)$.
\end{rem}

\section{The $\mathfrak{M}_H(G)$-property.} \label{section:MHG} 
Let $G$ be a uniform $p$-group, and let $X_\infty/X$ be a voltage cover to a voltage assignment ${\alpha \colon S \longrightarrow G}$. The $\mathfrak{M}_H(G)$-property was a crucial assumption in the previous section. In this section we will give necessary conditions for this property to hold.

Let as before $H\subseteq G$ be a normal subgroup such that ${G/H\cong \Z_p}$. In this section we study the $\mathfrak{M}_H(G)$-property for $\textup{Pic}_\infty$ and $J_\infty$. In view of the exact sequence 
\[ 0 \longrightarrow J_\infty \longrightarrow \textup{Pic}_\infty \longrightarrow \Lambda/\textup{Aug} \longrightarrow 0, \] 
where $\textup{Aug}$ means the augmentation ideal of $\Lambda$, Lemma~\ref{lemma:MHG} implies that $J_\infty$ satisfies the $\mathfrak{M}_H(G)$-property if and only if $\textup{Pic}_\infty$ satisfies the $\mathfrak{M}_H(G)$-property. 

Let $\Lambda_1=\Z_p\llbracket G/H\rrbracket $. Let $Y$ be the unique Galois cover of $X$ such that 
$$\Gal(X_\infty/Y)=H$$ 
and ${\Gal(Y/X) \cong G/H}$. We write $\textup{Pr}_{\Lambda_1}$ for $\textup{Pr}(Y)\otimes \Lambda_1$, and we define $\textup{Div}_{\Lambda_1}$ analogously. We write $\textup{Pic}_Y$ for $\textup{Div}_{\Lambda_1}/\textup{Pr}_{\Lambda_1}$. We have already seen in Lemma \ref{lem:free} that $\textup{Div}_\Lambda$, $\textup{Pr}_\Lambda$, $\textup{Div}_{\Lambda_1}$ and $\textup{Pr}_{\Lambda_1}$ are free over $\Lambda$ and $\Lambda_1$, respectively. If one looks carefully at the generators of these modules one even sees that they are generated by the same sets of elements. 
\begin{lemma} \label{lemma:control} 
Let $\Omega\subseteq \Lambda$ be the kernel of the natural projection ${\Lambda\twoheadrightarrow \Lambda_1}$. Then we have an isomorphism
\[\textup{Pic}_\infty/\Omega\textup{Pic}_\infty\longrightarrow \textup{Pic}_Y. \]
\end{lemma}
\begin{proof}
Consider the commutative diagram
\begin{align*} 
\xymatrix{0\ar[r]&\textup{Pr}_\Lambda\ar[r]\ar[d]&\textup{Div}_\Lambda\ar[r]^\pi \ar[d] & \textup{Pic}_\infty\ar[d]^r\ar[r]&0 \\
0\ar[r]&\textup{Pr}_{\Lambda_1}\ar[r]&\textup{Div}_{\Lambda_1}\ar[r]& \textup{Pic}_Y\ar[r]&0.}  
\end{align*}

All vertical maps are surjective. Applying the snake lemma gives an exact sequence 
\[ 0 \longrightarrow {\Pr}_\Lambda \cap \Omega \cdot \Div_\Lambda \longrightarrow \Omega \cdot \Div_\Lambda \longrightarrow \ker(r) \longrightarrow 0.\] 
Therefore ${\ker(r) = \pi(\Omega \Div_\Lambda) =\Omega {\textup{Pic}}_\infty}$, and we obtain the claim.
\end{proof}
As an immediate corollary we obtain
\begin{corollary}
\label{cor:mhg}
If ${\mu_{\Lambda_1}(\textup{Pic}_Y) = 0}$, then the $\mathfrak{M}_H(G)$-property holds for ${\textup{Pic}_\infty}$. 
\end{corollary}
\begin{proof} 
  Since $\textup{Pic}_Y$ is a torsion $\Lambda_1$-module, the assumption ${\mu_{\Lambda_1}(\textup{Pic}_Y) = 0}$ implies that $\textup{Pic}_Y$ is finitely generated over $\Z_p$, and thus the same holds true for $\textup{Pic}_\infty/\Omega \textup{Pic}_\infty$ in view of Lemma~\ref{lemma:control}. By the definition of $\Omega$, it follows from Nakayama's Lemma~\ref{nakayama} that $\textup{Pic}_\infty$ is finitely generated as a $\Z_p\llbracket H\rrbracket $-module. 
\end{proof} 

In the setting of the above corollary, it follows from the proof that in fact ${\mu_\Lambda(\textup{Pic}_\infty) = 0 = \mu_{\Lambda_1}(Y)}$. We can prove a more general statement. 
\begin{thm} \label{thm:MHG-iff} 
   If the $\mathfrak{M}_H(G)$-property holds for $\textup{Pic}_\infty$, then 
   \[ \mu_{\Lambda}(\textup{Pic}_\infty) = \mu_{\Lambda_1}(\textup{Pic}_Y). \]
\end{thm} 
\begin{proof} 
  We follow the approach from the proof of \cite[Proposition~4.7]{Lim_MHG}, with necessary changes. 
  To ease notation, we let ${X = \textup{Pic}_\infty}$, ${X_f = X/X[p^\infty]}$. Starting from the exact sequence 
  \[ 0 \longrightarrow X[p^\infty] \longrightarrow X \longrightarrow X_f \longrightarrow 0, \] 
  we obtain a long exact sequence of $H$-homology groups: 
  \[ \ldots \longrightarrow H_{i+1}(H,X) \longrightarrow H_{i+1}(H, X_f) \longrightarrow H_i(H, X[p^\infty]) \longrightarrow H_i(H,X) \longrightarrow \ldots \] 
  \[ \ldots \longrightarrow H_1(H,X) \longrightarrow H_1(H, X_f) \longrightarrow H_0(H, X[p^\infty]) \longrightarrow H_0(H,X) \longrightarrow H_0(H, X_f) \longrightarrow 0.\] 

  By hypothesis, ${H_0(H, X_f) = (X_f)_H}$ is finitely generated over $\Z_p$ and therefore it is a torsion $\Lambda_1$-module. Since also $H_0(H,X[p^\infty])$ is finitely generated over $\Lambda_1$, and obviously also torsion, it follows that $H_0(H,X)$ is a finitely generated torsion $\Lambda_1$-module. 
  
  More generally, since $X_f$ is finitely generated over $\Z_p\llbracket H\rrbracket $ by Nakayama's Lemma, and as $H$ is a compact pro-$p$, $p$-adic Lie group without $p$-torsion, it follows from \cite[Lemma~3.2.3]{Lim-Sharifi} (see also \cite[proof of Theorem~1.1, p.~638]{howson}) that each $H_i(H,X_f)$ is finitely generated over $\Z_p$.  
  In particular, all the $H_i(H, X_f)$ are torsion $\Lambda_1$-modules. This implies that all the terms in the above long exact sequence are finitely generated and torsion $\Lambda_1$-modules. 
  
  It follows from the additivity of $\mu_{\Lambda_1}$-invariants on exact sequences of finitely generated torsion $\Lambda_1$-modules that 
  \[ \mu_{\Lambda_1}(H_i(H, X[p^\infty])) = \mu_{\Lambda_1}(H_i(H, X)) \] 
  for each ${i \in \N}$. 
  Using Lemma~\ref{lemma:mu}(b), we may conclude that 
  \begin{eqnarray*} \mu_\Lambda(X) & = & \sum_{i \ge 0} (-1)^i  \mu_{\Lambda_1}(H_i(H, X[p^\infty])) \\ & = & \mu_{\Lambda_1}(H_0(H, X)) + \sum_{i \ge 1} (-1)^{i} \mu_{\Lambda_1}(H_i(H, X)). \end{eqnarray*} 
  It follows from our control theorem (Lemma~\ref{lemma:control}) that the first summand equals ${\mu_{\Lambda_1}(\textup{Pic}_Y})$. On the other hand, the long exact $H$-homology sequence attached to the short exact sequence 
  \[ 0 \longrightarrow \Lambda^a \longrightarrow \Lambda^a \longrightarrow X \longrightarrow0 \]
  and the fact that $\Lambda$ is cohomologically trivial as a $\Z_p\llbracket H\rrbracket $-module because it is (compactly) induced imply that ${H_i(H,X) = 0}$ for each ${i \ge 1}$. This proves the theorem. 
\end{proof} 
In the special case where the dimension of $G$ is 2, we can prove an if-and-only-if statement. 
\begin{corollary} \label{cor:MHG-iff} 
   If ${H \cong \Z_p}$, then the following statements hold. \begin{compactenum}[(a)] 
      \item $\mu_{\Lambda}(\textup{Pic}_\infty) \le \mu_{\Lambda_1}(\textup{Pic}_Y)$. 
      \item The $\mathfrak{M}_H(G)$-property holds for $\textup{Pic}_\infty$ if and only if 
   \[ \mu_{\Lambda}(\textup{Pic}_\infty) = \mu_{\Lambda_1}(\textup{Pic}_Y). \] 
   \end{compactenum} 
\end{corollary} 
\begin{proof} 
  We use the notation from the proof of Theorem~\ref{thm:MHG-iff}. As we have seen in this proof we have ${H_i(H,X) = 0}$ for each ${i \ge 1}$. Therefore the long exact sequence from the proof implies that 
  \[ \mu_{\Lambda_1}(H_i(H, X[p^\infty])) = \mu_{\Lambda_1}(H_{i+1}(H, X_f))\]
  for each ${i \in \N}$. Now Lemma~\ref{lemma:mu}(b) implies that 
  \[ \mu_\Lambda(X) = \mu_{\Lambda_1}(H_0(H, X[p^\infty])) + \sum_{i \ge 1} (-1)^i \mu_{\Lambda_1}(H_{i+1}(H, X_f)). \]
  Since the cohomological dimension of $H$ is equal to 1, the right hand side of this formula reduces to 
  $\mu_{\Lambda_1}(H_0(H, X[p^\infty]))$. By the bottom of the long exact homology sequence this is equal to 
  \[ \mu_{\Lambda_1}(H_1(H, X_f)) + \mu_{\Lambda_1}(H_0(H,X)) - \mu_{\Lambda_1}(H_0(H, X_f)). \]
  The second summand is equal to $\mu_{\Lambda_1}(\textup{Pic}_Y)$ by the control theorem Lemma~\ref{lemma:control}. Moreover, since ${X_f \cong p^m X}$ for some ${m \ge 0}$, we can view $X_f$ as a submodule of $X$. Using that the cohomological dimension of $H$ is 1, the long exact homology sequence to the short exact sequence 
  \[ 0 \longrightarrow X_f \longrightarrow X \longrightarrow X/X_f \longrightarrow 0\] 
  implies that ${H_1(H, X_f) = 0}$ (because ${H_1(H, X) = 0}$ and ${H_2(H, X/X_f) = 0}$). 

  Therefore 
  \begin{align} \label{mu-formel} \mu_\Lambda(X) = \mu_{\Lambda_1}(\textup{Pic}_Y) - \mu_{\Lambda_1}(H_0(H, X_f)). \end{align} 
  This immediately proves assertion~(a) of the corollary. 
  Note that 
  $$\mu_{\Lambda_1}(H_0(H, X_f)) = \mu_{\Lambda_1}((X_f)_H) = 0$$ 
  if and only if $(X_f)_H$ is finitely generated over $\Z_p$. Nakayama's Lemma~\ref{nakayama} implies that this property in turn is equivalent to $X_f$ being finitely generated as a $\Z_p\llbracket H\rrbracket $-module. This concludes the proof of the corollary. 
\end{proof}

\subsection{Examples}
Let $X$ be the following graph
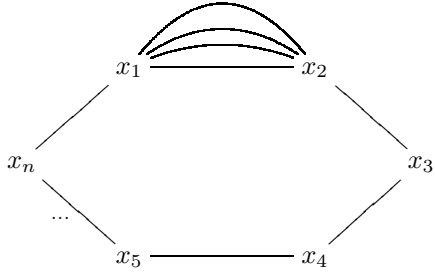
\[ \xymatrix{&&&&\\ & x_1 \ar@{-}[rr] \ar@{-}@/^1.2pc/[rr] \ar@{-}@/^0.7pc/[rr] \ar@{-}@/^2pc/[rr] & & x_2 \ar@{-}[rd] & \\ 
x_n \ar@{-}[ur] & & & & x_3 \ar@{-}[ld] \\ 
& x_5 \ar@{-}[rr] \ar@{-}[ul]^{\ldots} & & x_4 &} \]
In other words, we let ${X = X_0}$ be a graph with $n$ vertices ${x_1, \ldots, x_n}$, which are connected in a cycle, with one multiple edge, say, between $x_1$ and $x_2$. 
Let ${G=\Z_p^l\rtimes \Z_p}$. Let $\tau_1,\dots ,\tau_l$ be topological generators of $\Z_p^l$ and let $\tau_{l+1}$ be a topological generator of the factor $\Z_p$ such that $\tau_1,\dots,\tau_l,\tau_{l+1}$ are topological generators of $G$. Assume that we have at least $l+1$ edges between $x_1$ and $x_2$, i.e. the degree of $x_1$ is at least $l+2$. Let $\{e_1,\dots, e_{a-1}\}$ be the edges between $x_1$ and $x_2$. 

\begin{example} 
Let $H=\Z_p^l$. Fix a voltage assignment ${\alpha \colon S \longrightarrow G}$ with the property $\alpha(e_i)=\tau_i$ for $1\le i \le l+1$ and $\alpha(e_i)\in H$ for $l+2\le i\le a-1$. By Lemma~\ref{connected} each of the derived graphs $X_n$ is connected. Note that $\alpha$ induces a unique $\Z_p$-voltage assignment 
\[\alpha'\colon S\longrightarrow G\longrightarrow G/H\cong \Z_p,\]
which sends $\tau_1,\dots, \tau_l$ to the trivial element. Let $Y/X$ be the $\Z_p$-cover corresponding to $\alpha'$.  We want to find integers $a$ and $n$ such that $X_\infty$ satisfies the $\mathfrak{M}_H(G)$-property. By Corollary~\ref{cor:mhg} it suffices to show that $\mu(\textup{Pic}_Y)=0$. In \cite[Example~8.4]{Kleine-Mueller4} (we apply this example with the choice ${b = 1}$) we have shown that the characteristic ideal of $\textup{Pic}_Y$ is generated by
\[T^2((a-2)n+a-1).\]

Thus, the $\mathfrak{M}_H(G)$-property holds whenever $p$ does not divide $(a-2)n+a-1$. In particular, it holds if $p\mid (a-2)$. 
\end{example} 

We conclude with an example where we can prove that the $\mathfrak{M}_H(G)$-property holds in a ${\mu \ne 0}$-setting. 
\begin{example} 
  Suppose that ${G \cong \Z_p \rtimes \Z_p}$ is generated topologically by two elements $\sigma$ and $\tau$ and that $p>2$. Assume that $n=2$ and assume that we have $3p$ edges between $x_1$ and $x_2$. Suppose that we are given a voltage assignment ${\alpha \colon S \longrightarrow G}$ such that 
  \[ \alpha(e_1) = \ldots = \alpha(e_p) = \sigma, \quad \alpha(e_{p+1}) = \ldots = \alpha(e_{2p}) = \tau,\quad \alpha(e_{2p+1})=\ldots =\alpha(e_{3p})=1. \] Now suppose that ${H = \langle \sigma\rangle}$.
  By Lemma~\ref{connected} each of the derived graphs $X_n$ is connected. 
 The natural projection ${G \twoheadrightarrow G/H}$ induces a surjection ${\Lambda \twoheadrightarrow \Lambda_1}$, where we write ${\Lambda_1 := \Z_p\llbracket G/H\rrbracket \cong \Z_p\llbracket T\rrbracket }$.  Then the number of edges $e_i$ between $x_1$ and $x_2$ which are mapped to 1 under the composite voltage assignment 
  \[ \alpha' \colon S \longrightarrow G \twoheadrightarrow G/H\] 
  is equal to $2p$, and the remaining edges are mapped to a topological generator $\gamma$ of $G/H$, which we identify with ${T+1 \in \Lambda_1}$.  

  For the $\Z_p$-cover $Y$ with voltage assignment $\alpha'$, we obtain by the definitions the matrix 
  \[ D - A_{\alpha'} = \begin{pmatrix} 3p & -p \gamma - 2p\\ - p \gamma^{-1} - 2p& 3p \end{pmatrix}, \]
  where we recall that $\gamma$ denotes the image of $\tau$ under the canonical projection ${G \twoheadrightarrow G/H}$ (i.e. $\gamma$ is a topological generator of ${G/H \cong \Z_p}$). The determinant of this matrix is 
  \[ \det(\Delta_\infty) = 9p^2 - p^2 - 2p^2 \gamma^{-1} - 2p^2 \gamma - 4p^2 = 2p^2 (2 - \gamma - \gamma^{-1}). \]
  Recalling $T=\gamma-1$, this power series is associated to \[2p^2(2(T+1)-(T+1)^2-1)=-2p^2T^2. \]
 As $p\neq 2$, we may conclude that ${\mu_{\Lambda_1}(\textup{Pic}_Y) = 2}$. 

  In particular, it follows from Corollary~\ref{cor:MHG-iff}(a) that ${\mu_\Lambda(\textup{Pic}_\infty) \le 2}$. It therefore remains to prove that the $\mu$-invariant of $\textup{Pic}_\infty$ is at least 2. The corresponding matrix is 
  \[ D - A_\alpha = \begin{pmatrix} 3p & -p \tau - p \sigma-p \\ -p \tau^{-1} - p \sigma^{-1}-p & 2p \end{pmatrix} = p \cdot \begin{pmatrix} 3 & - \tau - \sigma -1\\ - \tau^{-1} - \sigma^{-1}-1 & 3 \end{pmatrix}.\] 

Recall that we have a presentation
\[ 0 \longrightarrow \Lambda^2 \stackrel{h}{\longrightarrow} \Lambda^2 \longrightarrow \textup{Pic}_\infty \longrightarrow 0, \] 
  and that $h$ is given by the matrix $D-A_\alpha^t$. As each entry of this matrix is divisible by $p$ we have a natural surjection
  \[\textup{Pic}_\infty \twoheadrightarrow \Lambda^2/p.\]
  The $\mu$-invariant of $\Lambda^2/p$ is clearly equal to $2$. As $\textup{Pic}_\infty$ is $\Lambda$-torsion, it follows from the additivity of $\mu$-invariants on short exact sequences of torsion $\Lambda$-modules that ${\mu_\Lambda(\textup{Pic}_\infty)\ge 2}$. Thus, ${2=\mu_\Lambda(\textup{Pic}_\infty)=\mu_{\Lambda_1}(\textup{Pic}_Y)}$. Then Corollary~\ref{cor:MHG-iff} implies that the $\mathfrak{M}_H(G)$-property holds for $\textup{Pic}_\infty$.
\end{example} 

\bibliography{references} 

\begin{thebibliography}{10}

\bibitem{brumer}
Armand Brumer.
\newblock Pseudocompact algebras, profinite groups and class formations.
\newblock {\em J. Algebra}, 4:442--470, 1966.

\bibitem{burns}
David Burns.
\newblock On main conjectures in non-commutative {I}wasawa theory and related
  conjectures.
\newblock {\em J. Reine Angew. Math.}, 698:105--159, 2015.

\bibitem{5people}
John Coates, Takako Fukaya, Kazuya Kato, Ramdorai Sujatha, and Otmar Venjakob.
\newblock The {$\rm GL_2$} main conjecture for elliptic curves without complex
  multiplication.
\newblock {\em Publ. Math. Inst. Hautes \'{E}tudes Sci.}, (101):163--208, 2005.

\bibitem{coates-et-all}
John Coates and Dohyeong Kim.
\newblock Introduction to the work of {M}. {K}akde on the non-commutative main
  conjectures for totally real fields.
\newblock In {\em Noncommutative {I}wasawa main conjectures over totally real
  fields}, volume~29 of {\em Springer Proc. Math. Stat.}, pages 1--22.
  Springer, Heidelberg, 2013.

\bibitem{corry-Perkinson}
Scott Corry and David Perkinson.
\newblock {\em Divisors and sandpiles}.
\newblock American Mathematical Society, Providence, RI, 2018.
\newblock An introduction to chip-firing.

\bibitem{curtis-reiner1}
Charles~W. Curtis and Irving Reiner.
\newblock {\em Methods of representation theory. {V}ol. {I}}.
\newblock Pure and Applied Mathematics. John Wiley \& Sons, Inc., New York,
  1981.
\newblock With applications to finite groups and orders.

\bibitem{Dixon}
John~D. Dixon, Marcus P.~F. du~Sautoy, Avinoam Mann, and Dan Segal.
\newblock {\em Analytic pro-{$p$} groups}, volume~61 of {\em Cambridge Studies
  in Advanced Mathematics}.
\newblock Cambridge University Press, Cambridge, second edition, 1999.

\bibitem{dubose-vallieres}
Sage {Dubose} and Daniel {Vallières}.
\newblock On $\mathbb{Z}_l^d$-towers of graphs.
\newblock {\em Algebraic Combinatorics}, ahead of print, 2023.

\bibitem{eisenbud}
David {Eisenbud}.
\newblock {\em {Commutative algebra. With a view toward algebraic geometry.}}
\newblock Berlin: Springer-Verlag, 1995.

\bibitem{ellerbrock-nickel}
Nils {Ellerbrock} and Andreas {Nickel}.
\newblock {Integrality of Stickelberger elements and annihilation of natural
  Galois modules}.
\newblock {\em Preprint}, 2023.

\bibitem{diss-gonet}
Sophia~R. Gonet.
\newblock {\em Jacobians of {F}inite and {I}nfinite {V}oltage {C}overs of
  {G}raphs}.
\newblock ProQuest LLC, Ann Arbor, MI, 2021.
\newblock Thesis (Ph.D.)--The University of Vermont and State Agricultural
  College.

\bibitem{gonet}
Sophia~R. Gonet.
\newblock Iwasawa theory of {Jacobians} of graphs.
\newblock {\em Algebr. Comb.}, 5(5):827--848, 2022.

\bibitem{goodearl-warfield}
Kenneth~R. Goodearl and Robert~B. Warfield, Jr.
\newblock {\em An introduction to noncommutative {N}oetherian rings}, volume~61
  of {\em London Mathematical Society Student Texts}.
\newblock Cambridge University Press, Cambridge, second edition, 2004.

\bibitem{howson}
Susan Howson.
\newblock {Euler characteristics as invariants of {I}wasawa modules}.
\newblock {\em Proc. London Math. Soc. (3)}, 85(3):634--658, 2002.

\bibitem{iwasawa}
Kenkichi {Iwasawa}.
\newblock On {$\Gamma $}-extensions of algebraic number fields.
\newblock {\em Bull. Amer. Math. Soc.}, 65:183--226, 1959.

\bibitem{kakde}
Mahesh Kakde.
\newblock The main conjecture of {I}wasawa theory for totally real fields.
\newblock {\em Invent. Math.}, 193(3):539--626, 2013.

\bibitem{Kleine-Mueller4}
Sören {Kleine} and Katharina {Müller}.
\newblock {On the growth of the Jacobians in $\Z_p^l$-voltage covers of
  graphs}.
\newblock {\em Preprint}, 2023.

\bibitem{kundu-lim}
Debanjana Kundu and Meng~Fai Lim.
\newblock Control theorems for fine {S}elmer groups.
\newblock {\em J. Th\'{e}or. Nombres Bordeaux}, 34(3):851--880, 2022.

\bibitem{Lam}
Tsit~Y. Lam.
\newblock {\em A first course in noncommutative rings}, volume 131 of {\em
  Graduate Texts in Mathematics}.
\newblock Springer-Verlag, New York, second edition, 2001.

\bibitem{lazard}
Michel Lazard.
\newblock Groupes analytiques {$p$}-adiques.
\newblock {\em Inst. Hautes \'{E}tudes Sci. Publ. Math.}, (26):389--603, 1965.

\bibitem{Lim_MHG}
Meng~Fai Lim.
\newblock A remark on the {$\mathfrak{M}_H(G)$}-conjecture and {A}kashi series.
\newblock {\em Int. J. Number Theory}, 11(1):269--297, 2015.

\bibitem{Lim_MHG2}
Meng~Fai Lim.
\newblock {$\mathfrak{M}_H(G)$}-property and congruence of {G}alois
  representations.
\newblock {\em J. Ramanujan Math. Soc.}, 33(1):37--74, 2018.

\bibitem{Lim-Sharifi}
Meng~Fai Lim and Romyar~T. Sharifi.
\newblock Nekov{\'a}r duality over {{\(p\)}}-adic {Lie} extensions of global
  fields.
\newblock {\em Doc. Math.}, 18:621--678, 2013.

\bibitem{mazur-wiles}
Barry Mazur and Andrew Wiles.
\newblock Class fields of abelian extensions of {{\(\mathbb Q\)}}.
\newblock {\em Invent. Math.}, 76:179--330, 1984.

\bibitem{vallieres2}
Kevin McGown and Daniel Valli\`eres.
\newblock {On abelian {$\ell$}-towers of multigraphs II}.
\newblock {\em Ann. Math. Qu\'{e}.}

\bibitem{vallieres3}
Kevin McGown and Daniel Valli\`eres.
\newblock {On abelian {$\ell$}-towers of multigraphs III}.
\newblock {\em Ann. Math. Qu\'{e}.}

\bibitem{nickel-definition}
Andreas Nickel.
\newblock Non-commutative {F}itting invariants and annihilation of class
  groups.
\newblock {\em J. Algebra}, 323(10):2756--2778, 2010.

\bibitem{nickel}
Andreas Nickel.
\newblock Notes on noncommutative {F}itting invariants.
\newblock In {\em Development of {I}wasawa theory---the centennial of {K}.
  {I}wasawa's birth}, volume~86 of {\em Adv. Stud. Pure Math.}, pages 27--60.
  Math. Soc. Japan, Tokyo, [2020] \copyright 2020.
\newblock With an appendix by Henry Johnston and Nickel.

\bibitem{northcott}
D.~G. Northcott.
\newblock {\em Finite free resolutions}, volume No. 71.
\newblock Cambridge University Press, Cambridge-New York-Melbourne, 1976.

\bibitem{ochi-venjakob2}
Yoshihiro Ochi and Otmar Venjakob.
\newblock On the structure of {S}elmer groups over {$p$}-adic {L}ie extensions.
\newblock {\em J. Algebraic Geom.}, 11(3):547--580, 2002.

\bibitem{ritter-weiss2}
J\"{u}rgen Ritter and Alfred Weiss.
\newblock Toward equivariant {I}wasawa theory. {II}.
\newblock {\em Indag. Math. (N.S.)}, 15(4):549--572, 2004.

\bibitem{ritter-weiss}
J\"{u}rgen Ritter and Alfred Weiss.
\newblock On the ``main conjecture'' of equivariant {I}wasawa theory.
\newblock {\em J. Amer. Math. Soc.}, 24(4):1015--1050, 2011.

\bibitem{roquette}
Peter Roquette.
\newblock Realisierung von {D}arstellungen endlicher nilpotenter {G}ruppen.
\newblock {\em Arch. Math. (Basel)}, 9:241--250, 1958.

\bibitem{serre}
Jean-Pierre Serre.
\newblock Sur la dimension cohomologique des groupes profinis.
\newblock {\em Topology}, 3:413--420, 1965.

\bibitem{skinner-urban}
Christopher Skinner and Eric Urban.
\newblock The {Iwasawa} {Main} {Conjectures} for {{\(\mathrm{GL}_{2}\)}}.
\newblock {\em Invent. Math.}, 195(1):1--277, 2014.

\bibitem{sunada}
Toshikazu Sunada.
\newblock {\em Topological crystallography}, volume~6 of {\em Surveys and
  Tutorials in the Applied Mathematical Sciences}.
\newblock Springer, Tokyo, 2013.
\newblock With a view towards discrete geometric analysis.

\bibitem{vallieres1}
Daniel Valli\`eres.
\newblock On abelian {$\ell$}-towers of multigraphs.
\newblock {\em Ann. Math. Qu\'{e}.}, 45(2):433--452, 2021.

\bibitem{venjakob}
Otmar Venjakob.
\newblock On the structure theory of the {I}wasawa algebra of a {$p$}-adic
  {L}ie group.
\newblock {\em J. Eur. Math. Soc. (JEMS)}, 4(3):271--311, 2002.

\bibitem{k-book}
Charles~A. Weibel.
\newblock {\em The {$K$}-book}, volume 145 of {\em Graduate Studies in
  Mathematics}.
\newblock American Mathematical Society, Providence, RI, 2013.
\newblock An introduction to algebraic $K$-theory.

\end{thebibliography}
\bibliographystyle{plain}

\end{document}